\documentclass[11pt,twoside]{amsart}
  \calclayout % Setting margins of the paper
\usepackage{amsmath}
\usepackage{amssymb}
\usepackage{amsthm}
\usepackage{amsfonts}
\usepackage[nobysame]{amsrefs}
\usepackage[english]{babel}
\usepackage{bbold}
\usepackage{enumitem}
\setlist[description]{leftmargin=\parindent}
\usepackage{esint}
\usepackage{graphicx}
\usepackage[colorlinks=true,citecolor=red,linkcolor=blue]{hyperref}
\usepackage{todonotes}

\numberwithin{equation}{section}

\newtheorem{theorem}{Theorem}[section]
\newtheorem{definition}[theorem]{Definition}
\newtheorem{lemma}[theorem]{Lemma}
\newtheorem{proposition}[theorem]{Proposition}
\newtheorem{corollary}[theorem]{Corollary}
\newtheorem{notation}[theorem]{Notation}
\newtheorem{remark}[theorem]{Remark}

\newcommand{\diam}{\operatorname{diam}}
\newcommand{\inndiam}{\operatorname{inn\,diam}}
\newcommand{\dist}{\operatorname{dist}} 
 
\newcommand{\Exc}{\operatorname{Exc}}

\newcommand{\supp}{\operatorname{supp}}

\newcommand{\ol}{\overline}

\newcommand{\wt}{\widetilde}
\newcommand{\coleq}{\mathrel{\mathop:}=}
\newcommand{\eqcol}{=\mathrel{\mathop:}}

\definecolor{lightred}{rgb}{1,.5,.5}
\definecolor{lightblue}{rgb}{.5,.5,1}

\renewcommand{\d}{\mathrm d}
\newcommand{\R}{\mathbb R}

\title[Random elliptic operators with degenerate coefficients]{Regularity of random elliptic operators with degenerate coefficients and applications to stochastic homogenization}
\author[P. Bella]{Peter Bella}
\address{Peter Bella \hfill\break
	Faculty of Mathematics, TU Dortmund University, 
	Vogelpothsweg 87, 44227 Dortmund, Germany}
\email{peter.bella@tu-dortmund.de}
\author[M. Kniely]{Michael Kniely}
\address{Michael Kniely \hfill\break
	Faculty of Mathematics, TU Dortmund University, 
	Vogelpothsweg 87, 44227 Dortmund, Germany}
\email{michael.kniely@tu-dortmund.de}
\curraddr{Weierstrass Institute for Applied Analysis and Stochastics, 
	Mohrenstrasse 39, 10117 Berlin, Germany}
\email{michael.kniely@wias-berlin.de}

\begin{document}
	\begin{abstract}
		We consider degenerate elliptic equations of second order in divergence form with a symmetric random coefficient field $a$. Extending the work of the first author, Fehrman, and Otto [Ann.\ Appl.\ Probab.\ \textbf{28} (2018), no.\ 3, 1379--1422], who established the large-scale $C^{1,\alpha}$ regularity of $a$-harmonic functions in a degenerate situation, we provide stretched exponential moments for the minimal radius $r_*$ describing the minimal scale for this $C^{1,\alpha}$ regularity. As an application to stochastic homogenization, we partially generalize results by Gloria, Neukamm, and Otto [Anal.\ PDE \textbf{14} (2021), no.\ 8, 2497--2537] on the growth of the corrector, the decay of its gradient, and a quantitative two-scale expansion to the degenerate setting. On a technical level, we demand the ensemble of coefficient fields to be stationary and subject to a spectral gap inequality, and we impose moment bounds on $a$ and $a^{-1}$. We also introduce the ellipticity radius $r_e$ which encodes the minimal scale where these moments are close to their positive expectation value. 
	\end{abstract}
	\maketitle
	%\vfill
	%\tableofcontents
	%\vfill

\section{Introduction and main results}
In these notes, we present some ideas to generalize results from stochastic homogenization of uniformly elliptic operators $-\nabla \cdot a \nabla$ to the case of degenerate and unbounded random coefficient fields $a$. The underlying random distribution is always assumed to be stationary and ergodic. To quantify the degeneracy and unboundedness, we impose moment bounds on the norm of $a(x)$ and its pointwise inverse $a(x)^{-1}$, $x \in \mathbb R^d$. A precise collection of our general assumptions is given in Definition \ref{defadmissible}. 

For the sake of a simplified notation, we focus on scalar models where $a : \mathbb R^d \rightarrow \mathbb R^{d \times d}$ is a matrix field rather than a field of rank-4 tensors. But since we do not rely on results from scalar PDE theory like maximum principles, we believe that our methods also extend to systems provided that all arguments involving $|a|$ or $|a^{-1}|$ also apply to the respective generalizations $\mu$ and $\lambda^{-1}$ as defined by the first author, Fehrman, and Otto \cite{BFO18}. As we shall explain in more detail below (see Remark \ref{rem:a12}), we currently have to restrict ourselves for technical reasons to symmetric matrix coefficient fields $a(x) = a(x)^T$; this issue might be resolved by working only with the scalar quantities $\mu$ and $\lambda^{-1}$. However, this is beyond the scope of this contribution and is left as a subject for future work. 
An alternative approach for non-symmetric coefficient fields $a(x)$ is outlined in \cite[Chapter 10]{AKM19}, where a variational formulation based on a ``double-variable'' approach is presented. 

To some extent, this paper continues the studies of the first author, Fehrman, and Otto \cite{BFO18}, where the large-scale $C^{1,\alpha}$ regularity and a first-order Liouville principle for $a$-harmonic functions were derived in the same setting. It is one of the goals of the present contribution to provide stretched exponential moments for the minimal radius $r_*$, which determines the minimal scale for the $C^{1,\alpha}$ regularity. Moreover, we provide quantitative estimates on the growth of the corrector and the decay of its gradient, and we derive a quantitative two-scale expansion in our degenerate setting. 

The starting point of our analysis is the work of Gloria, Neukamm, and Otto \cite{GNO20} on the large-scale regularity of random elliptic operators. The main achievements of this publication are large-scale Schauder and large-scale Calder\'on--Zygmund estimates valid on scales larger than the minimal radius $r_*$. Their approach is in turn motivated by the ideas of Avellaneda and Lin \cite{AL87}, who established a large-scale regularity theory for elliptic operators with periodic coefficients, hence, on the torus. This enabled the authors to apply compactness arguments which are generally not available. %However, their philosophy is to establish a regularity theory for the operator $-\nabla \cdot a \nabla$ by transferring regularity results from the homogenized operator $-\nabla \cdot a_\mathrm{hom} \nabla$ to the original one. 
Previous preprints \cite{GNO20v2,GNO20v3} of \cite{GNO20} follow in some cases different strategies which can be equally valuable as they are sometimes better suited for an application in our situation. A key ingredient in all three versions are functional inequalities (e.g.\ spectral gap and logarithmic Sobolev inequalities), which allow to quantify certain aspects of the random ensemble in an advantageous manner. A comparison of various forms of functional inequalities and applications is given by Duerinckx and Gloria \cite{DG20}. The basis for our results on the corrector in stochastic homogenization and the two-scale expansion is the contribution by Gloria, Neukamm, and Otto \cite{GNO21} on quantitative estimates in stochastic homogenization.

\begin{definition}[Ensemble of coefficient fields]
\label{defadmissible}
Let $\Omega$ be the space of symmetric coefficient fields $a : \mathbb R^d \rightarrow \mathbb R^{d \times d}$, $d \geq 2$, and let $\langle \cdot \rangle$ denote an \emph{ensemble} of coefficient fields $a$, i.e.\ a probability measure on $\Omega$, which we assume to be 
\begin{itemize}
	\item \emph{stationary}, i.e.\ the probability distributions of $a$ and $a(x + \cdot)$ coincide for all $x \in \mathbb R^d$,
	\item \emph{ergodic}, i.e.\ every translation invariant random variable is almost surely constant.
\end{itemize}
For any $a \in \Omega$, we define the (space-dependent) quantities 
\[
\lambda \coleq \big| a^{-1} \big|^{-1} %\inf_{\xi \in \mathbb R^d} \frac{\xi \cdot a \xi}{|\xi|^2} 
\qquad \text{and} \qquad \mu \coleq |a|. %\sup_{\xi \in \mathbb R^d} \frac{|a\xi|^2}{\xi \cdot a \xi}.
\]
We suppose $\lambda(x), \mu(x) \in (0,\infty)$ for a.e.\ $x \in \mathbb R^d$, and that $p, q \in (1, \infty)$ exist satisfying 
\begin{align}
\label{eq:defk}
\big\langle \mu^p \big\rangle^\frac1p + \big\langle \lambda^{-q} \big\rangle^\frac1q \eqcol K < \infty \qquad \text{and} \qquad \frac1p + \frac 1q < \frac 2d,
\end{align}
where \eqref{eq:defk} is independent of $x \in \mathbb R^d$ due to the stationarity of the ensemble $\langle \cdot \rangle$. 
\end{definition}

The concept of imposing stochastic moment bounds on the coefficient field $a(x)$ instead of assuming uniform ellipticity was successfully applied in a similar context by Chiarini and Deuschel to prove an invariance principle for symmetric diffusion processes on $\mathbb R^d$ \cite{CD16}. 
In the context of homogenization, condition \eqref{eq:defk} was first imposed by Andres, Deuschel, and Slowik \cite{ADS15} for an ergodic random conductance model and later also used in a time-dependent ergodic version thereof \cite{DS16}. In our situation, the purpose of \eqref{eq:defk} is to guarantee the sublinearity of the corrector (cf.\ Remark \ref{remarkcorrectors}) and to allow for specific Sobolev embeddings (e.g.\ in Lemma \ref{lemmagrowth}). Only recently, the first author and Sch\"affner \cite{BS21} showed that the relaxed version 
\begin{align}
\label{eq:pq_optimal}
\frac{1}{p} + \frac{1}{q} < \frac{2}{d-1}
\end{align}
guarantees local boundedness and the existence of a Harnack inequality for solutions to linear, nonuniformly elliptic equations. The same result was already proven by Trudinger \cite{Tru71} under the more restrictive version in \eqref{eq:defk}. Condition \eqref{eq:pq_optimal} is, in addition, optimal in the sense that local boundedness is generally not available if the right-hand side is replaced by $\tfrac{2}{d-1} + \varepsilon$ for any $\varepsilon > 0$; we refer to the references in \cite{BS21} for further details. As an application to stochastic homogenization, the authors show that the pointwise sublinearity of the corrector, which was proven by Chiarini and Deuschel \cite{CD16} in a similar framework assuming \eqref{eq:defk}, also holds under condition \eqref{eq:pq_optimal}. 
Stochastic moment bounds of the type \eqref{eq:defk} appeared recently also in studies on the regularity properties of non-uniformly parabolic operators; see \cite{BS22, ACS21} and the references therein. %and, in particular, random conductance models in degenerate and unbounded settings.

Related to Birkhoff's ergodic theorem (see e.g.\ \cite{Kre85}) guaranteeing that 
\begin{align}
\label{eq:ergodic}
\lim_{r \rightarrow \infty} \ \Big( \fint_{B_{r}(0)} \mu^p \Big)^{\frac 1p} + \Big( \fint_{B_{r}(0)} \lambda^{-q} \Big)^{\frac 1q} = K
\end{align}
for a.e.\ coefficient field $a$, we subsequently define the \emph{ellipticity radius} $r_e$, which determines the minimal scale on which the system behaves approximately elliptic.

\begin{definition}
\label{def:ellipticity_radius}
For $K$ as in \eqref{eq:defk}, we define the \emph{ellipticity radius} $r_e \geq 1$ as the random variable 
\begin{align*}
r_e \coleq \inf \bigg\{ r \geq 1 \: \Big| \: \forall \, \rho > r : \Big( \fint_{B_{\rho}(0)} \mu^p \Big)^{\frac 1p} + \Big( \fint_{B_{\rho}(0)} \lambda^{-q} \Big)^{\frac 1q} \le 4K \bigg\}.
\end{align*}
%where $K$ is defined in \eqref{eq:defk}.
\end{definition}

We subsequently recall standard notions in stochastic homogenization including the extended corrector $(\phi, \sigma) = ((\phi_i)_i, (\sigma_{ijk})_{ijk})$ and the homogenized field $a_\mathrm{hom}$. Existence and uniqueness of the extended corrector will be discussed afterwards.

\begin{definition}[Definition of the extended corrector $(\phi, \sigma)$] 
\label{defcorrector}
In the situation of Definition \ref{defadmissible} and for given $\xi \in \mathbb R^d$, one calls the sublinear solution $\phi_\xi$ of $-\nabla \cdot a (\xi + \nabla \phi_\xi) = 0$ on $\mathbb R^d$ the \emph{corrector} associated to $\xi$. Specifically for $\xi$ being a canonical basis vector, one considers the corrector $\phi_i$, $1 \leq i \leq d$, being a solution to
\begin{align}
\label{eq:corr_phi}
-\nabla \cdot q_i = 0, \qquad q_i \coleq a ( e_i + \nabla \phi_i ).
\end{align}
The vector $q_i$ is called the $i$th component of the \emph{flux} and one introduces the \emph{flux correction} $\sigma_{ijk}$, $1 \leq i,j,k \leq d$, as a vector-valued potential solving 
\begin{align}
\label{eq:corr_sigma}
\nabla \cdot \sigma_i = q_i - \langle q_i \rangle, \qquad -\Delta \sigma_i = \nabla \times q_i \coleq (\partial_j q_{ik} - \partial_k q_{ij})_{jk}.
\end{align}
Finally, one defines the \emph{homogenized field} $a_\mathrm{hom}$ via $a_\mathrm{hom} e_i \coleq \langle q_i \rangle$.
\end{definition}

Concerning the possible degeneracy and unboundedness of the coefficients $a(x)$, we mention that it is obviously not possible to perform estimates like $c |v|^2 \leq v \cdot a v \leq C |v|^2$ for $v \in \mathbb R^d$ with uniform constants $C \geq c > 0$. It is, therefore, advantageous to introduce a separate notation for such quadratic forms and also for matrix products $a^\frac12 M a^\frac 12$ with some $M \in \mathbb R^{d \times d}$.

\begin{notation}
\label{not:a12}
For $a \in \mathbb R^{d \times d}$, $M \in \mathbb R^{d \times d}$, and $v \in \mathbb R^d$, we set 
$
|v|_{a}^2 \coleq v \cdot a v %\quad \text{and} \quad |v|_{a^{-1}}^2 \coleq v \cdot \big( a^S \big)^{-1} v. 
$
and 
$|M|_a \coleq \big|a^\frac12 M a^\frac12\big|$ where $|\cdot|$ denotes the spectral norm on $\mathbb R^{d \times d}$. 
For any measurable $D \subset \mathbb R^d$, we further abbreviate $L^2_a(D)^d \coleq \big\{ f : D \rightarrow \mathbb R^d \; \big| \; f \cdot a f \in L^1(D) \big\}$ and $H^1_a(D) \coleq \big\{ u : D \rightarrow \mathbb R \; \big| \; \nabla u \cdot a \nabla u \in L^1(D) \big\}$.
\end{notation}

A question which typically arises in this context, is concerned with the so-called Liouville principle. For example, given a subquadratic solution $u$ of $-\nabla \cdot a \nabla u = 0$ on $\mathbb R^d$, can one prove that $u = c + \xi \cdot x + \phi_\xi(x)$ for some $\xi \in \mathbb R^d$? For the present setting, the first author, Fehrman, and Otto \cite{BFO18} have shown that such a Liouville property does hold (cf. Remark \ref{remarkcorrectors}). The interest in such a principle also lies in its close relation to Schauder estimates, which has been highlighted by Simon \cite{Sim97}. Moreover, Liouville properties have been established in many situations including stationary and ergodic degenerate systems \cite{BDCKY15}, %higher-order regularity, higher-order correctors, 
higher-order Liouville principles \cite{FO16}, and Liouville theorems for uniformly parabolic systems in a random setting supposing stationarity and ergodicity \cite{BCF19}. 

The existence of an extended corrector $(\phi, \sigma)$ as in Definition \ref{defcorrector} directly follows from \cite[Lemma 1]{BFO18}, while its uniqueness is an immediate consequence of the sublinearity of $(\phi, \sigma)$ \cite[Lemma 2]{BFO18} and a related Liouville principle \cite[Theorem 1]{BFO18}. For the sake of completeness, we recall these results in Remark \ref{remarkcorrectors} below. 
Identity \eqref{eq:sublinear_correctors} shows that the extended corrector $(\phi, \sigma)$ is sublinear w.r.t.\ $\rho$. While \eqref{eq:sublinear_correctors} gives only a qualitative statement, we will prove a quantified version thereof in Corollary \ref{corgrowth}, which will serve as an important tool in the latter part of this paper.

\begin{remark}[Properties of the extended corrector $(\phi, \sigma)$ {\cite[Lemmas 1,2, Theorem 1]{BFO18}, \cite[Proposition 4.1]{CD16}}]
\label{remarkcorrectors}
Under the hypotheses of Definition \ref{defadmissible}, there exist a constant $C > 0$ and random tensor fields $\phi_i$ and $\sigma_{ijk}$, $1 \leq i,j,k \leq d$, satisfying \eqref{eq:corr_phi}--\eqref{eq:corr_sigma} and the skew-symmetry $\sigma_{ijk} = -\sigma_{ikj}$, while the gradient fields are stationary, of vanishing expectation 
\[
\big\langle \nabla \phi_i \big\rangle = \big\langle \nabla \sigma_{ijk} \big\rangle = 0,
\]
and having bounded moments 
\[
\sum_{i=1}^d \big\langle |\nabla \phi_i|_a^2 \big\rangle + \sum_{i=1}^d \Big\langle |\nabla \phi_i|^\frac{2q}{q+1} \Big\rangle^\frac{q+1}{2q} + \sum_{i,j,k=1}^d \Big\langle |\nabla \sigma_{ijk}|^\frac{2p}{p+1} \Big\rangle^\frac{p+1}{2p} \leq C K,
\]
where $K$ is the constant from Definition \ref{defadmissible}. In addition, $(\phi, \sigma)$ is sublinear in the sense 
\begin{align}
\label{eq:sublinear_correctors}
\lim_{\rho \rightarrow \infty} \max \bigg\{ \frac1{\rho} \Big( \fint_{B_\rho} \Big| \phi - \fint_{B_\rho} \phi \Big|^\frac{2p}{p-1} \Big)^\frac{p-1}{2p}, \, \frac1{\rho} \Big( \fint_{B_\rho} \Big| \sigma - \fint_{B_\rho} \sigma \Big|^\frac{2q}{q-1} \Big)^\frac{q-1}{2q} \bigg\} = 0, 
\end{align}
and a.e.\ coefficient field $a$ satisfies the following Liouville principle: Any solution $u \in H^1_a(\mathbb R^d)$ to $-\nabla \cdot a \nabla u = 0$ in $\mathbb R^d$ subject to 
\[
\lim_{R \rightarrow \infty} R^{-(1 + \alpha)} \Big( \fint_{B_R} |u|^\frac{2p}{p-1} \Big)^\frac{p-1}{2p} = 0
\]
for some $\alpha \in (0, 1)$ admits the representation $u(x) = c + \xi \cdot x + \phi_\xi(x)$ for some $c \in \mathbb R$ and $\xi \in \mathbb R^d$. 
Finally, the homogenized field $a_\mathrm{hom}$ is uniformly elliptic. 
\end{remark}

The following lemma is basically a consequence of the collected results in Remark \ref{remarkcorrectors} and \cite[Theorem 2]{BFO18}. The claimed mean-value property in \eqref{eq:meanvalues} is a direct consequence of \eqref{eqexcessdecay} which is typically referred to as a large-scale $C^{1,\alpha}$ regularity estimate. The first large-scale regularity result for a uniformly elliptic, scalar equation was obtained by Marahrens and Otto \cite{MO15}, where the ergodicity of the random ensemble was encoded by means of a logarithmic Sobolev inequality. For elliptic systems with stationary and coercive coefficients, the first author and Otto \cite{BO16} derived moment bounds on the corrector gradient by employing either a logarithmic Sobolev inequality or a spectral gap estimate. 
We also mention the large-scale regularity theory for scalar equations in a random environment developed by Armstrong and Smart \cite{AS16}. A crucial ingredient of their approach is the assumption of a finite range of dependence for the symmetric coefficient field. 
More recently, large-scale regularity results have also been shown for the random conductance model by Armstrong and Dario \cite{AD18}. They prove that the corresponding solutions on supercritical percolation clusters are close to harmonic functions on large scales which admit stretched exponential moments. For similar models subject to long-range correlations and decoupling inequalities, Sapozhnikov \cite{Sap17} generalized several results, such as heat kernel bounds and parabolic Harnack inequalities, which have already been known for the Bernoulli percolation.

\begin{lemma}[Large-scale $C^{1,\alpha}$ regularity and a mean-value property for $a$-harmonic functions]
\label{lemmac1alpha}
For any $\alpha \in (0, 1)$ and $K > 0$, there exist constants $C_0$, $C_1$, and $C_2$ such that for all positive radii $r < R$ and $p,q \in (1,\infty)$ satisfying $\frac1p + \frac 1q \leq \frac 2d$ the following holds: If 
\begin{align*}
\Big( \fint_{B_{\rho}} \mu^p \Big)^{\frac 1p} + \Big( \fint_{B_{\rho}} \lambda^{-q} \Big)^{\frac 1q} \le 2K
\end{align*}
and 
\begin{align*}
\max \bigg\{ \frac1{\rho} \Big( \fint_{B_\rho} \Big| \phi - \fint_{B_\rho} \phi \Big|^\frac{2p}{p-1} \Big)^\frac{p-1}{2p}, \, \frac1{\rho} \Big( \fint_{B_\rho} \Big| \sigma - \fint_{B_\rho} \sigma \Big|^\frac{2q}{q-1} \Big)^\frac{q-1}{2q} \bigg\} \leq \frac{1}{C_0}
\end{align*}
%where $\phi$ and $\sigma$ are the correctors from Remark \ref{remarkcorrectors}, 
for all $\rho \in [r,R]$, then any solution $u \in H^1_a(B_R)$ of $-\nabla \cdot a \nabla u = 0$ fulfills the excess-decay 
\begin{align}
\label{eqexcessdecay}
\Exc (r) \leq C_1 \Big( \frac r R \Big)^{2\alpha} \Exc (R),
\end{align}
%for a constant $C>0$ independent of $r$ and $R$ 
where the excess $\Exc(\rho)$ is defined as 
\[
\Exc(\rho) \coleq \inf_{\xi \in \mathbb R^d} \fint_{B_\rho} \big| \nabla u - (\xi + \nabla \phi_\xi) \big|_a^2,
\]
and $\nabla u$ satisfies the mean-value property 
\begin{align}
\label{eq:meanvalues}
\fint_{B_r} |\nabla u|_a^2 \leq C_2 \fint_{B_R} |\nabla u|_a^2.
\end{align}
\end{lemma}

We now introduce the \emph{minimal radius} $r_*$ which quantifies the minimal scale on which the (extended) corrector $(\phi, \sigma)$ grows only sublinearly. For technical reasons, we do not only demand $r_* \geq r_e$ but even $r_* \geq M_0 r_e$ for a specific constant $M_0 \geq 1$ detailed below.

\begin{definition}[Minimal radius]
\label{defminradius}
In the situation of Definition \ref{defadmissible}, we define the \emph{minimal radius} as the random variable $r_* \geq M_0 r_e$ given in the form 
\[
r_* \coleq \inf \bigg\{ r \geq M_0 r_e \: \Big| \: \forall \, \rho > r : \max \bigg\{ \frac1{\rho} \Big( \fint_{B_\rho} \Big| \phi - \fint_{B_\rho} \phi \Big|^\frac{2p}{p-1} \Big)^\frac{p-1}{2p}, \frac1{\rho} \Big( \fint_{B_\rho} \Big| \sigma - \fint_{B_\rho} \sigma \Big|^\frac{2q}{q-1} \Big)^\frac{q-1}{2q} \bigg\} \leq \frac{1}{C_0} \bigg\}
\]
where $C_0$ is the constant from Lemma \ref{lemmac1alpha}, while $M_0 \geq 1$ is defined in \eqref{eqdefm0}.
\end{definition}

The following spectral gap estimate \eqref{eq:sgineq} is our main stochastic assumption on the underlying random environment. A very similar condition involving a coarsening partition $\{D\}$ of $\mathbb R^d$ was used in \cite{GNO20v2}. Alternatively, one can employ multiscale functional inequalities to describe the random ensemble (see Remark \ref{rem:multiscale}). A detailed exposition of these ideas is given by Duerinckx and Gloria in \cite{DG20}. We point out that the more elementary spectral gap condition \eqref{eq:sgineq} is sufficient for the present study, where we deduce stretched exponential moments for $r_*$ with a typically small exponent $\varepsilon$. Nevertheless, we remark that multiscale functional inequalities could provide a framework to obtain (with different techniques) stronger stretched exponential bounds on $r_*$ in the spirit of Gloria, Neukamm, and Otto \cite{GNO20}.

\begin{definition}[Spectral gap inequality]
\label{deffuncineq}
Let the hypotheses of Definition \ref{defadmissible} hold and assume that a partition $\{D\}$ of $\mathbb R^d$ and an exponent $\beta \in [0, 1)$ exist such that 
\begin{align}
\label{eq:beta}
\diam D \leq (\dist D + 1)^\beta \leq C(d) \inndiam D, 
\end{align}
where $\inndiam D \coleq 2 \sup \big\{ r \geq 0 \ \big| \ \exists \; x \in D : B_r(x) \subset D \big\}$ denotes the \emph{inner diameter} of $D \subset \mathbb R^d$.

We say that a random field $a$ satisfies the \emph{spectral gap} (or \emph{Poincar\'e}) \emph{inequality}, if there exists a constant $\kappa \in (0,1]$ such that 
\begin{align}
\label{eq:sgineq}
\Big\langle \big( X(a) - \langle X(a) \rangle \big)^2 \Big\rangle \leq \frac 1\kappa \bigg\langle \sum_D \Big( \int_D \Big| \frac{\partial X(a)}{\partial a} \Big|_a \Big)^2 \bigg\rangle
\end{align}
for all $\sigma(a)$-measurable random variables $X(a)$, where we recall that 
\begin{align*}
\int_D \Big| \frac{\partial X(a)}{\partial a} \Big|_a = \sup_{\| b\|_{L^\infty(D)} = 1} \int_D b : a^{\frac 12} \frac{\partial X(a)}{\partial a} a^{\frac 12} = \sup_{\| b\|_{L^\infty(D)} = 1} \limsup_{t\rightarrow 0} \frac{X \big(a + t a^{\frac 12} b a^{\frac 12} \big) - X(a)}{t}.
\end{align*}
\end{definition}

An upgraded version of the standard spectral gap estimate \eqref{eq:sgineq} to higher order moments will be provided in Lemma \ref{lemma:psgineq}, which will serve as a useful tool in various situations subsequently in this paper. 

\begin{remark}
\label{rem:multiscale}
Instead of \eqref{eq:sgineq}, one can also use a multiscale spectral gap inequality 
\begin{align}
\label{eq:multiscale}
\Big\langle \big(X(a) - \langle X(a) \rangle \big)^2 \Big\rangle \leq \bigg\langle \int_0^\infty \int_{\mathbb R^d} \Big( \int_{B_\ell(x)} \Big| \frac{\partial X(a)}{\partial a} \Big| \Big)^2 \, \d x \, \frac{\pi(\ell)}{(\ell+1)^d} \, \d\ell \bigg\rangle
\end{align}
or (multiscale) logarithmic Sobolev inequalities for quantifying the ergodicity of the ensemble $\langle \cdot \rangle$. The weight function $\pi : [0,\infty) \to [0,\infty)$ in \eqref{eq:multiscale} is generally assumed to be integrable. Integrable correlations $\mathrm{Cov} ( a(x); a(0) )$ can be modeled with weights decaying like $\pi(\ell) \sim (\ell+1)^{-1-\alpha}$ for $\alpha > d$. We refer to the work of Duerinckx and Gloria \cite{DG20} for further details. %In the case of a scalar Gaussian field $a(x)$ (which we apart from here do not consider in this article), $\alpha > d(1-\beta)$ (along with $\pi(\ell) \sim -\gamma'(\ell)$ and $| \mathrm{Cov} ( a(x); a(0) ) | \leq \gamma(|x|)$) is required in \cite{GNO20v3} to guarantee the existence of a partition $\{D\}$ such that \eqref{eq:beta} and a coarsened logarithmic Sobolev inequality hold. 
\end{remark}

The connection between the ellipticity radius $r_e$ and the stochastic integrability of the underlying coefficient field $a$ satisfying the spectral gap inequality \eqref{eq:sgineq} is clarified in the following lemma. We shall basically prove that stretched exponential moment bounds on averages of $|a|^p$ and $|a^{-1}|^q$ carry over to $r_e$.

\begin{lemma}[Stretched exponential moments for $r_e$]
	\label{lemma:re_moments}
	Assume that an ensemble of coefficient fields $a \in \Omega$ is given according to Definition \ref{defadmissible}, which satisfies the spectral gap estimate \eqref{eq:sgineq} along with $\beta \in [0, 1)$ subject to \eqref{eq:beta}. In case that $\int_{B_1} \lambda^{-q}$ and $\int_{B_1} \mu^p$ possess stretched exponential moments 
	\begin{align}
	\label{eq:strexpmom_mu} 
	\max \bigg\{ \Big\langle \exp \Big( \frac1C \Big( \int_{B_1} \lambda^{-q} \Big)^\alpha \Big) \Big\rangle, \, \Big\langle \exp \Big( \frac1C \Big( \int_{B_1} \mu^p \Big)^\alpha \Big) \Big\rangle \bigg\} < 2
	\end{align} 
	for some constants $\alpha > 0$ and $C > 0$, then the ellipticity radius $r_e$ from Definition \ref{def:ellipticity_radius} is subject to 
	\begin{align}
	\label{eq:strexpmom_re}
	\Big\langle \exp \Big( \frac1C r_e^{\frac{\alpha}{\alpha + 1} \frac{d}{2} (1-\beta)} \Big) \Big\rangle < 2
	\end{align}
	with the same parameter $\alpha > 0$ but a possibly different constant $C > 0$. 
	
%	If $\lambda$ and $\mu$ even allow for an almost sure bound 
%	\begin{align}
%	\label{eq:alsurebound_mu}
%	\max \bigg\{ \int_{B_1} \lambda^{-q}, \int_{B_1} \mu^p \bigg\} < C 
%	\end{align}
%	for some $C > 0$, then we obtain (with a generally different $C > 0$) 
%	\begin{align}
%	\label{eq:alsurebound_re}
%	\Big\langle \exp \Big( \frac1C r_e^{\frac{d}{2} (1-\beta)} \Big) \Big\rangle < 2.
%	\end{align}
	
	%where (with $\epsilon > 0$ arbitrarily small)
	%\[
	%\alpha_* \coleq 
	%\begin{cases}
	%\alpha, &\alpha < d, \\ 
	%d - \epsilon, &\alpha = d, \\ 
	%d, &\alpha > d.
	%\end{cases}
	%\]
\end{lemma}

We are now in a position to show that the minimal radius $r_*$ introduced in Definition \ref{defminradius} possesses stretched exponential moments by adapting the line of arguments from Gloria, Neukamm, and Otto \cite{GNO20v2}. In contrast to the final version \cite{GNO20} of the aforementioned preprint, optimal stochastic integrability is not achieved in \cite{GNO20v2}. The main tool which allows the authors to improve the stochastic integrability of $r_*$ is a modified extended corrector $(\phi_T, \sigma_T)$ living on the length scale $\sqrt{T}$ and arising from a ``massive approximation''. As we are currently not able to adapt this approach to our situation, we resort to the more elementary approach in the preprint \cite{GNO20v2} where a bound on $r_*$ similar to the one in \eqref{eq:strexpmom_rast} is obtained.

In a nutshell, the problem arises from the ``massive term'' $\frac1T (\phi_T, \sigma_T)$ in the following system for $(\phi_T, \sigma_T)$ (cf.\ \cite[(37)--(39)]{GNO20}): 
\begin{align*}
\frac1T \phi_T - \nabla \cdot a ( \nabla \phi_T + e ) &= 0, \\ 
\frac1T \sigma_T - \Delta \sigma_T &= \nabla \times q_T, \qquad q_T \coleq a ( \nabla \phi_T + e ).
\end{align*}
The additional massive term gives rise to an exponential localization of $(\phi_T, \sigma_T)$ at the length scale $\sqrt{T}$. Related to that, the authors of \cite{GNO20} repeatedly work with convolutions of $\nabla \phi_T$ with a Gaussian distribution $G_T$ of variance $T$; see e.g.\ \cite[Proposition 3, Lemma 2]{GNO20} and in particular the proofs thereof. In our situation, where we want to formulate estimates also in terms of weighted $L^2_a$ and $L^2_{a^{-1}}$ norms, we have to face the following structural obstacle: Given a function $f \in L^2_a(\mathbb R^d)^d$, we generally end up with $f \ast G_T \notin L^2_a(\mathbb R^d)^d$. In other words, the space $L^2_a(\mathbb R^d)^d$ is not invariant under such a convolution.

\begin{theorem}[Stretched exponential moments for $r_*$]
\label{theoremmoments}
Suppose that the hypotheses of Definition \ref{defadmissible} on the ensemble of coefficient fields $a \in \Omega$ satisfying the spectral gap inequality \eqref{eq:sgineq} hold together with $\beta \in [0,1)$ subject to \eqref{eq:beta}. Moreover, assume that $\int_{B_1} \lambda^{-q}$ and $\int_{B_1} \mu^p$ allow for the stretched exponential moments in \eqref{eq:strexpmom_mu} with $\alpha \coleq \frac{\varepsilon}{1 - \varepsilon}$ where $\varepsilon \in (0,1)$ is the hole-filling exponent from Proposition \ref{propsensitivity}.

Then, the minimal radius $r_*$ as defined in Definition \ref{defminradius} fulfills 
\begin{align}
\label{eq:strexpmom_rast}
\Big\langle \exp \Big( \frac1C r_*^{\varepsilon \frac{d}{2} (1-\beta)} \Big) \Big\rangle < 2
\end{align}
for a sufficiently large constant $C > 0$.
\end{theorem}

An essential part of the proof of Theorem \ref{theoremmoments} is concerned with the sensitivity analysis quantifying the dependence of $\nabla (\phi, \sigma)$ on the coefficient field $a$. At the end, we need to control averages of $\nabla (\phi, \sigma)$ on balls around the origin, but we shall give a slightly more general statement below. As above, the massive $(\phi_T, \sigma_T)$-regularization prevents us from proceeding as in \cite{GNO20}. However, we could prove an analogue of the statement in the intermediate version \cite{GNO20v3}, but as Theorem \ref{theoremmoments} is already posed in the language of \cite{GNO20v2}, it suffices to generalize the sensitivity result in \cite{GNO20v2} to our setting.

\begin{proposition}[Sensitivity estimate for average integrals]
\label{propsensitivity}
Let the assumptions of Definition \ref{defadmissible} on the ensemble of coefficient fields $a \in \Omega$ be in place, and let a partition $\{D\}$ of $\mathbb R^d$ and $\beta \in [0,1)$ be given according to \eqref{eq:beta}. Consider the linear functional 
\[
F\psi = \int g \cdot \psi
\]
acting on vector fields $\psi: \mathbb R^d \rightarrow \mathbb R^d$, where $g: \mathbb R^d \rightarrow \mathbb R^d$ is supported in $B_r$ for some radius $r \geq r_e$. 

Then, there exist a hole-filling exponent $\varepsilon = \varepsilon(d,K) \in (0,1)$ and a constant $C = C(d,K) > 1$ such that for any $g$ as above satisfying 
\begin{align}
\label{eq:averfunc}
\max \bigg\{ \Big( \fint_{B_r} |g|^\frac{2p}{p-1} \Big)^\frac{p-1}{2p}, \, \Big( \fint_{B_r} |g|^\frac{2q}{q-1} \Big)^\frac{q-1}{2q} \bigg\} \lesssim r^{-d},
\end{align}
the following bound on the functional derivative of $F$ holds: 
\begin{align}
\label{eq:sensitivity}
\sum_D \Big( \int_D \Big| \frac{\partial F \nabla \phi}{\partial a} \Big|_a \Big)^2 + \sum_D \Big( \int_D \Big| \frac{\partial F \nabla \sigma_{jk}}{\partial a} \Big|_a \Big)^2 \leq C \Big( \frac{(r+r_*)^{1-\varepsilon(1-\beta)}}{r} \Big)^d.
\end{align}
%\begin{equation*}
%\ell^{-d} \int  |\partial^\textrm{fct}_{x,\ell} F \nabla \psi|^2 \leq C \ell^{\frac{d(P-1)}{P(P+1)}+d} \Big( 1 + \frac{\ell}{r_* + r} \Big)^\frac{2\alpha}{P+1} r^{-d(2-\frac{1}{P})} \bigg(\int \omega^\frac{P+1}{P-1} \bigg)^\frac{P-1}{P+1}
%\end{equation*}
%where
%\[
%\omega(x) \coleq \big( \ell^{-d} r_*(x)^{(1 - \varepsilon)d} r_e(x)^{\varepsilon d} + 1 \big) \Big(\frac{r + r_*}{|x| + r + r_*}\Big)^{\frac{2\alpha}{P+1}}.
%\]
\end{proposition}

We now employ the above results on the existence of stretched exponential moments for the minimal radius $r_*$ (cf.\ Theorem \ref{theoremmoments}) and the sensitivity estimate for (extended) corrector gradients (cf.\ Proposition \ref{propsensitivity}) to derive quantitative estimates on the decay of the corrector gradient $\nabla (\phi, \sigma)$ and the growth of the corrector itself. Due to the relatively weak ($\varepsilon$-dependent) stretched exponential moments available for $r_*$ (compared to \cite{GNO20} and its preprint \cite{GNO20v3}), the subsequent results also involve a dependence on $\varepsilon$.

\begin{theorem}[Decay of the corrector gradient and growth of the corrector]
\label{theoremcorrectors}
Assume that the ensemble of coefficient fields $a \in \Omega$ fulfills the assumptions of Definition \ref{defadmissible} and satisfies the spectral gap estimate \eqref{eq:sgineq} along with $\beta \in [0, 1)$ subject to \eqref{eq:beta}. Let $\varepsilon \in (0, 1)$ denote the constant from Proposition \ref{propsensitivity}. 

Then, there exists a stationary random field $\mathcal C(x)$ with stretched exponential moments 
\begin{align}
\label{eq:randomfield_corrector}
\Big\langle \exp \Big( \frac1C \mathcal C^{\varepsilon (1 - \beta)} \Big) \Big\rangle < 2
\end{align}
for a sufficiently large constant $C > 0$ such that the following assertions hold: 
\begin{enumerate}
\item If $m : \mathbb R^d \rightarrow \mathbb R^d$ is bounded and supported in $B_r$, $r \geq 1$, $\fint_{B_r} |m|^2 = 1$, and assumption \eqref{eq:strexpmom_mu} holds true with $\alpha \coleq \frac{\varepsilon}{1 - \varepsilon}$, then, for all $x \in \mathbb R^d$, 
\[
\bigg| \fint_{B_r} \nabla (\phi, \sigma) (x+y) \cdot m(y) \, dy \, \bigg| \leq \mathcal C(x) r^{-\frac \varepsilon 2 d (1 - \beta)}.
\]
\item %Assume now that $\int_{B_1} \lambda^{-q}$ and $\int_{B_1} \mu^p$ are almost surely bounded as stated in \eqref{eq:alsurebound_mu} 
If $\varepsilon \in \big( 0, \frac{\alpha}{\alpha + 1} - \frac{1}{\min\{p,q\}} \big]$ and \eqref{eq:strexpmom_mu} holds for some $\alpha > \frac{1}{\min\{p,q\} - 1}$, then the correctors $(\phi, \sigma)$ fulfill 
\[
\Big( \fint_{B_1(x)} | \phi |^\frac{2p}{p-1} \Big)^\frac{p-1}{2p} + \Big( \fint_{B_1(x)} | \sigma |^\frac{2q}{q-1} \Big)^\frac{q-1}{2q} \lesssim \Big| \fint_{B_1} (\phi, \sigma) \Big| + \mathcal C(x) \pi(|x|)
\]
together with  
\begin{align}
\label{eq:defpi}
\pi(r) \coleq 
\begin{cases}
1, &0 \leq \beta < 1 - \frac {2}{\varepsilon d}, \\ 
\log(2 + r), &\beta = 1 - \frac {2}{\varepsilon d}, \\ 
r^{\frac{\varepsilon d}{2} (\frac{2}{\varepsilon d} - 1 + \beta)}, &\beta > 1 - \frac {2}{\varepsilon d}. 
\end{cases}
\end{align}
\end{enumerate}
\end{theorem}

Our last result gives a quantitative estimate for a two-scale expansion. It is mainly a consequence of Theorem \ref{theoremcorrectors} on the growth of the corrector and the stochastic integrability of the random field $\mathcal C$ in \eqref{eq:randomfield_corrector}. We formulate the statement in the same spirit as in \cite{GNO21}; in particular, we employ the same averaging procedure over small balls for reasons of generality (even though this might not be necessary in many cases). But in contrast to \cite{GNO21}, we again encounter the small parameter $\varepsilon$ (coming from Theorem \ref{theoremcorrectors}), and we also get an additional term on the right-hand side of \eqref{eq:twoscaleerror} (which can be (formally) absorbed in the other term on the right-hand side in the limit $q \rightarrow \infty$).

\begin{corollary}[Quantitative two-scale expansion]
\label{cor:twoscale}
Suppose that the ensemble of coefficient fields $a \in \Omega$ meets the requirements of Definition \ref{defadmissible} and fulfills the spectral gap estimate \eqref{eq:sgineq}. Besides, let \eqref{eq:beta} hold with $\beta \in [0, 1)$, suppose that assumption \eqref{eq:strexpmom_mu} with $\alpha > \frac{1}{\min\{p,q\} - 1}$ is in place, and let the hole-filling exponent from Proposition \ref{propsensitivity} be restricted to $\varepsilon \in \big(0, \frac{\alpha}{\alpha + 1} - \frac{1}{\min\{p,q\}} \big]$. 
For $R \geq r_e$ and $\delta > 0$, let $g \in W^{1,\frac{2q}{q-1}}(\mathbb R^d)$ be supported in $B_R$, and let $u_\delta$ and $u_\mathrm{hom}$ denote the solutions to 
\[
-\nabla \cdot a \big( \tfrac{\cdot}{\delta} \big) \nabla u_\delta = \nabla \cdot g, \qquad -\nabla \cdot a_\mathrm{hom} \nabla u_\mathrm{hom} = \nabla \cdot g, 
\]
while the error $z_\delta$ in the two-scale expansion and the small-scale average $u_{\mathrm{hom},\delta}$ are defined by 
\[
z_\delta \coleq u_\delta - \big( u_{\mathrm{hom},\delta} + \delta \phi_i \big( \tfrac{\cdot}{\delta} \big) \partial_i u_{\mathrm{hom},\delta} \big), \qquad u_{\mathrm{hom},\delta}(x) \coleq \fint_{B_\delta(x)} u_\mathrm{hom}.
\]
We then have 
\begin{align}
\label{eq:twoscaleerror}
\Big( \int \big| \nabla z_\delta \big|_a^2 \Big)^\frac12 \lesssim \delta^{1+\frac{d}{2q}} \Big( \int |\nabla g|^\frac{2q}{q-1} \Big)^\frac{q-1}{2q} + \mathcal C_{\delta, g} \delta \pi(\delta^{-1}) \Big( \int \pi(|x|)^2 |\nabla g|^2 \Big)^\frac12
\end{align}
where $\pi(r)$ is defined in \eqref{eq:defpi} and where the random field $\mathcal C_{\delta, g}$ satisfies 
\begin{align}
\label{eq:quant-two-scale}
\bigg\langle \exp \bigg( \frac1C \mathcal C_{\delta, g}^{\big( 1 + \frac{\alpha + 1}{\alpha} \frac{\varepsilon}{\min\{p,q\}} \big)^{-1} \varepsilon (1 - \beta)} \bigg) \bigg\rangle < 2
\end{align}
for a sufficiently large constant $C > 0$ independent of $\delta$, $g$, $p$, and $q$. 
\end{corollary}

We conclude this section with a remark on the relations between the constants introduced above. In particular, we show that all conditions imposed on the constants are indeed feasible. 

\begin{remark}
From the definition of the hole-filling exponent $\varepsilon \in (0,1)$ in terms of the constant $E = C_d C_\mathrm{Sob}^2 C_\mathrm{Poi}^2 K^2$ in Step 1 of the proof of Proposition \ref{propsensitivity} (cf.\ \eqref{eq:defholefill}) we see that 
\begin{align}
\label{eq:estholefill}
1 + \frac1E = 2^{\varepsilon d} \leq 1 + \frac{2^d - 1}{d} \varepsilon d \leq 1 + 2^d \varepsilon \quad \Longrightarrow \quad \varepsilon \geq \frac{1}{C_d C_\mathrm{Sob}^2 C_\mathrm{Poi}^2 K^2}
\end{align}
for an adapted constant $C_d \geq 1$. Thus, \eqref{eq:estholefill} provides an initial lower bound for $\varepsilon \in (0, 1)$ only in terms of the constant $K \geq 1$ from \eqref{eq:defk} and domain-dependent constants $C_\mathrm{Sob} \geq 1$ and $C_\mathrm{Poi} \geq 1$. 

The requirements on $\alpha > 0$ and $\varepsilon \in (0, 1)$ in Theorem \ref{theoremcorrectors} and Corollary \ref{cor:twoscale} should be seen as compatibility conditions for the orders of stochastic integrability of the minimal radius $r_\ast$ and the ellipticity radius $r_e$. On the one hand, $\alpha > \frac{1}{\min\{p,q\} - 1}$ is needed to ensure the positivity of $\frac{\alpha}{\alpha + 1} - \frac{1}{\min\{p,q\}}$. On the other hand, $\varepsilon \in \big( 0, \frac{\alpha}{\alpha + 1} - \frac{1}{\min\{p,q\}} \big]$ guarantees that certain powers of $r_\ast$ and $r_e$ are stochastically integrable with the \emph{same order} at the end of the proof of Theorem \ref{theoremcorrectors}. Here, we use the fact that $\varepsilon \in (0, 1)$ can indeed be chosen sufficiently small since \eqref{eqholefill} remains true for smaller $\varepsilon$. Finally, the order of stochastic integrability of the random fields $\mathcal C_{\delta, g}$ in \eqref{eq:quant-two-scale} lies in the interval $(0, \varepsilon (1 - \beta))$. 
\end{remark}

\section{Large-scale $C^{1,\alpha}$ regularity quantified by the minimal radius $r_\ast$}
\subsection{Proof of Lemma \ref{lemmac1alpha}: A mean-value property for $a$-harmonic functions} 
\begin{proof}[Proof of Lemma \ref{lemmac1alpha}]
We divide the proof into two steps. First, we derive a non-degeneracy property for $\xi + \nabla \phi_\xi$ with $\xi \in \mathbb R^d$, while the desired mean-value property is proven as a consequence in the second step.  

\emph{Step~1. Excess decay and non-degeneracy.} Under the hypotheses of the lemma, we may apply \cite[Theorem 2]{BFO18} to establish \eqref{eqexcessdecay}. 
Note that we subsequently use \eqref{eqexcessdecay} with the choice $\alpha \coleq \frac12$. Following \cite{GNO20}, we shall first prove a non-degeneracy condition for the correctors $\phi_\xi$ in the sense 
\begin{align}
\label{eqnondegeneracy}
c |\xi|^2 \leq \fint_{B_r} \big|\xi + \nabla \phi_\xi \big|_a^2 \leq C |\xi|^2
\end{align}
for all $r \geq r_*$ and $\xi \in \mathbb R^d$ where $0 < c < C$ are independent of $r$ and $\xi$. For the lower bound, we first recall the elementary bound 
\[
\Big( \int_{B_r} |\xi + \nabla \phi_\xi|^\frac{2q}{q+1} \Big)^\frac{q+1}{q} \leq \Big( \int_{B_r} \lambda^{-q} \Big)^\frac1q \int_{B_r} |\xi + \nabla \phi_\xi|_a^2.
\]
Together with Poincar\'e's inequality, we derive 
\[
\Big( \fint_{B_r} |\xi + \nabla \phi_\xi|_a^2 \Big)^\frac12 \gtrsim \Big( \fint_{B_r} |\xi + \nabla \phi_\xi|^\frac{2q}{q+1} \Big)^\frac{q+1}{2q} \gtrsim \frac1r \Big( \fint_{B_r} |\xi \cdot x + \phi_\xi - \fint_{B_r} \phi_\xi|^\frac{2q}{q+1} \Big)^\frac{q+1}{2q}.
\]
The triangle inequality, Jensen's inequality, and the sublinear growth of the corrector now yield 
\begin{align*}
\Big( \fint_{B_r} |\xi + \nabla \phi_\xi|_a^2 \Big)^\frac12 &\gtrsim \frac1r \Big( \fint_{B_r} |\xi \cdot x|^\frac{2q}{q+1} \Big)^\frac{q+1}{2q} - \frac1r \Big( \fint_{B_r} \Big|\phi_\xi - \fint_{B_r} \phi_\xi \Big|^\frac{2q}{q+1} \Big)^\frac{q+1}{2q} \\ 
&\geq \frac{|\xi|}{r} \Big( \fint_{B_r} |x|^\frac{2q}{q+1} \Big)^\frac{q+1}{2q} - \frac{|\xi|}{r} \Big( \fint_{B_r} \Big|\phi - \fint_{B_r} \phi \Big|^\frac{2p}{p-1} \Big)^\frac{p-1}{2p} \gtrsim |\xi| - \frac{1}{C_0} |\xi|
\end{align*}
taking the scaling of $\phi_\xi$ and the notation $\phi = (\phi_i)_i$ into account. Choosing the constant $C_0 > 0$ sufficiently large, one arrives at the desired lower bound in \eqref{eqnondegeneracy}. Similarly, the upper bound is a consequence of the Caccioppoli estimate (carried out e.g.\ in \cite[Lemma 3]{BFO18})
\[
\int_{B_r} |\xi + \nabla \phi_\xi|_a^2 \leq \frac{4}{r^2} \Big( \int_{B_{2r}} \mu^p \Big)^\frac1p \Big( \int_{B_{2r}} \Big|\xi \cdot x + \phi_\xi - \fint_{B_{2r}} \phi_\xi \Big|^\frac{2p}{p-1} \Big)^\frac{p-1}{p}. % \Big( \fint_{B_{2r}} |a|^p \Big)^\frac1p.
\]
By the same reasoning as above, we obtain 
\[
\Big( \fint_{B_r} |\xi + \nabla \phi_\xi|_a^2 \Big)^\frac12 \lesssim \frac{|\xi|}{r} \Big( \fint_{B_{2r}} |x|^\frac{2p}{p-1} \Big)^\frac{p-1}{2p} + \frac{|\xi|}{r} \Big( \fint_{B_{2r}} \Big|\phi - \fint_{B_{2r}} \phi \Big|^\frac{2p}{p-1} \Big)^\frac{p-1}{2p} \lesssim |\xi| + \frac{1}{C_0} |\xi|.
\]
The claimed bound \eqref{eqnondegeneracy} now follows. 

\emph{Step~2. Mean-value property.} The ideas of \cite{GNO20} also apply to our situation, but we present the main steps for completeness. The lower bound in \eqref{eqnondegeneracy} ensures for any $\rho \in [r_*, R]$ the existence of a unique $\xi_\rho \in \mathbb R^d$ such that  
\begin{align}
\label{eqxirhodefinition}
\Exc(\rho) = \fint_{B_\rho} \big| \nabla u - (\xi_\rho + \nabla \phi_{\xi_\rho}) \big|_a^2.
\end{align}
For radii $\rho, \rho' \in [r_*, R]$ satisfying $0 < \rho' - \rho \leq \rho$, we deduce by virtue of \eqref{eqnondegeneracy}, the linearity of $\xi \mapsto \phi_\xi$, and the triangle inequality 
\[
|\xi_\rho - \xi_{\rho'}|^2 \lesssim \fint_{B_\rho} \big|\xi_\rho - \xi_{\rho'} + \nabla \phi_{\xi_\rho - \xi_{\rho'}} \big|_a^2 \lesssim \fint_{B_\rho} \big|\nabla u - (\xi_\rho + \nabla \phi_{\xi_\rho})\big|_a^2 + \fint_{B_\rho} \big|\nabla u - (\xi_{\rho'} + \nabla \phi_{\xi_{\rho'}})\big|_a^2.
\]
Due to the minimality property \eqref{eqxirhodefinition} of $\xi_\rho$ and $\rho < \rho' \leq 2\rho$, this entails  
\begin{align}
\label{eqxirhosmall}
|\xi_\rho - \xi_{\rho'}|^2 \lesssim \fint_{B_\rho} \big|\nabla u - (\xi_{\rho'} + \nabla \phi_{\xi_{\rho'}})\big|_a^2 \lesssim \Exc(\rho').
\end{align}
For arbitrary $R \geq r \geq r_*$, we let $N \in \mathbb N$ be the integer such that $2^{-(N+1)} R < r \leq 2^{-N} R$, which allows us to use \eqref{eqxirhosmall} and \eqref{eqexcessdecay} (with $\alpha = \frac12$) to estimate 
\begin{align}
\label{eqxirholarge}
|\xi_r - \xi_R|^2 \leq \bigg( \sum_{n = 0}^N |\xi_{2^{-(n+1)} R} - \xi_{2^{-n} R}| \bigg)^2 \lesssim \bigg( \sum_{n = 0}^N 2^{-\frac{n}{2}} \sqrt{\Exc(R)} \bigg)^2 \lesssim \Exc(R).
\end{align}
By means of \eqref{eqxirhodefinition}, \eqref{eqnondegeneracy}, and \eqref{eqxirholarge}, we thus get 
\[
\fint_{B_r} |\nabla u|_a^2 \lesssim \Exc(r) + |\xi_r|^2 \leq \Exc(r) + \Exc(R) + |\xi_R|^2.
\]
Moreover, \eqref{eqexcessdecay} and the definition of the excess ensure $\Exc(r) \lesssim \Exc(R) \leq \fint_{B_R} |\nabla u|_a^2$, while 
\[
|\xi_R|^2 \lesssim \fint_{B_R} \big|\xi_R + \nabla \phi_{\xi_R} \big|_a^2 \lesssim \Exc(R) + \fint_{B_R} |\nabla u|_a^2
\]
is a result of \eqref{eqnondegeneracy} and \eqref{eqxirhodefinition}. This concludes the argument. 
\end{proof}

\subsection{Proof of Lemma \ref{lemma:re_moments}: Stretched exponential moments for $r_e$} 

\begin{lemma}[$P$th power spectral gap estimate]
	\label{lemma:psgineq}
	Let the ensemble of coefficient fields $a \in \Omega$ satisfy the assumptions in Definition \ref{defadmissible} and the spectral gap estimate \eqref{eq:sgineq} with an arbitrary partition $\{D\}$ of $\mathbb R^d$. Then, there exists a constant $C > 0$ such that 
	\begin{align}
	\label{eq:pspectralgap}
	\big\langle (\zeta - \langle \zeta \rangle)^{2P} \big\rangle^\frac1P \leq \frac{CP^2}{\kappa} \bigg\langle \Big( \sum_D \Big( \int_D \Big| \frac{\partial \zeta}{\partial a} \Big|_a \Big)^2 \Big)^P \bigg\rangle^\frac1P 
	\end{align} 
	for any random variable $\zeta$ and all $P \in \mathbb N$, $P \geq 2$.
\end{lemma}
\begin{proof}
	The arguments are basically the same as in \cite{GNO20v2} but adapted to our degenerate setting. Applying the spectral gap estimate \eqref{eq:sgineq} to $\zeta^P$, we first derive 
	\[
	\langle \zeta^{2P} \rangle \leq \langle \zeta^{P} \rangle^2 + \frac1\kappa \bigg\langle \sum_D \Big( \int_D \Big| \frac{\partial \zeta^P}{\partial a} \Big|_a \Big)^2 \bigg\rangle.
	\]
	Elementary calculus guarantees that 
	\[
	\sum_D \Big( \int_D \Big| \frac{\partial \zeta^P}{\partial a} \Big|_a \Big)^2 = P^2 \zeta^{2(P-1)} \sum_D \Big( \int_D \Big| \frac{\partial \zeta}{\partial a} \Big|_a \Big)^2,
	\]
	while H\"older's inequality on the level of the probability measure $\langle \cdot \rangle$ yields 
	\[
	\Big\langle \sum_D \Big( \int_D \Big| \frac{\partial \zeta^P}{\partial a} \Big|_a \Big)^2 \Big\rangle \leq P^2 \langle \zeta^{2P} \rangle^{1-\frac1P} \bigg\langle \Big( \sum_D \Big( \int_D \Big| \frac{\partial \zeta}{\partial a} \Big|_a \Big)^2 \Big)^P \bigg\rangle^\frac1P.
	\]
	Young's inequality now allows to get rid of $\langle \zeta^{2P} \rangle$ on the right-hand side and to derive 
	\begin{align}
	\label{eq:pspectralgap_zetap}
	\langle \zeta^{2P} \rangle \leq C \langle \zeta^{P} \rangle^2 + \Big( \frac{CP^2}{\kappa} \Big)^P \bigg\langle \Big( \sum_D \Big( \int_D \Big| \frac{\partial \zeta}{\partial a} \Big|_a \Big)^2 \Big)^P \bigg\rangle
	\end{align}
	with some constant $C>0$. We now argue how to replace $\langle \zeta^P \rangle^2$ by $\langle \zeta^2 \rangle^P$ on the right-hand side. To this end, one writes $\zeta^P = \zeta^{P \frac{P-2}{P-1}} \zeta^{P \frac{1}{P-1}}$ and applies H\"older's inequality with exponents $2\frac{P-1}{P-2}$ and $2\frac{P-1}{P}$ followed by Young's inequality leading to 
	\[
	\langle \zeta^P \rangle^2 \leq \langle \zeta^{2P} \rangle^\frac{P-2}{P-1} \langle \zeta^2 \rangle^\frac{P}{P-1} \leq \frac1C \langle \zeta^{2P} \rangle + C^{P-2} \langle \zeta^2 \rangle^P
	\]
	with another constant $C>0$. %The first term on the right-hand side is absorbed on the left-hand side of \eqref{eq:pspectralgap_zetap}, while we 
	Using again the original spectral gap inequality and noting that it suffices to prove \eqref{eq:pspectralgap} for the case $\langle \zeta \rangle = 0$, we further obtain 
	\[
	\langle \zeta^2 \rangle^P \leq %CP \langle \zeta \rangle^{2P} + 
	\frac{1}{\kappa^P} \Big\langle \sum_D \Big( \int_D \Big| \frac{\partial \zeta}{\partial a} \Big|_a \Big)^2 \Big\rangle^P.
	\]
	The proof is finished taking Jensen's inequality $\langle ( \cdot ) \rangle^P \leq \langle ( \cdot )^P \rangle$ into account.
\end{proof}

\begin{proof}[Proof of Lemma \ref{lemma:re_moments}]
	We divide the proof into two steps. 
	
	\medskip
	\emph{Step~1. Exponential concentration for $\fint_{B_R} \mu^p$ and $\fint_{B_R} \lambda^{-q}$.}
	We start by recalling the upgraded spectral gap estimate from \eqref{eq:pspectralgap} and by applying it to $X(a) \coleq \fint_{B_R} \mu^p$ for some arbitrary $R \geq 1$. The same arguments are also applicable to $\fint_{B_R} \lambda^{-q}$. This yields % In the general case $\langle X(a) \rangle \neq 0$, the last estimate turns into 
	\begin{align}
	\label{eq:2rmoment}
	\Big\langle \big( X(a) - \langle X(a) \rangle \big)^{2r} \Big\rangle^\frac{1}{2r} \leq C r \bigg\langle \Big( \sum_D \Big( \int_D \Big| \frac{\partial X(a)}{\partial a} \Big|_a \Big)^2 \Big)^r \bigg\rangle^\frac{1}{2r}
	\end{align}
	for all $r \in \mathbb N$, $r \geq 2$. In a similar setting, exponential concentration and stretched exponential moments were shown in \cite[Proposition 1.10]{DG20} for arbitrary random variables $X(a)$ by assuming a deterministic bound of the form 
	\begin{align*}
	%\label{eq:sgbound}
	\sum_D \Big( \int_D \Big| \frac{\partial X(a)}{\partial a} \Big|_a \Big)^2 \leq \ol C
	\end{align*}
	(albeit employing a multiscale spectral gap inequality). Such a deterministic bound cannot be expected in our situation, instead we shall prove that 
	\begin{align}
	\label{eq:sgbound}
	\bigg\langle \Big( \sum_D \Big( \int_D \Big| \frac{\partial X(a)}{\partial a} \Big|_a \Big)^2 \Big)^r \bigg\rangle^\frac{1}{2r} \lesssim r^\frac{1}{\alpha} R^{-\frac{d}{2} (1 - \beta)}
	\end{align}
	holds true where $\lesssim$ means $\leq$ up to the prescribed parameters $d$, $p$, $q$, and $K$. To this end, we first calculate 
	\begin{align*}
	\sum_D \Big( \int_D \Big| \frac{\partial X(a)}{\partial a} \Big|_a \Big)^2 &= \sum_D \bigg( \sup_{\| b\|_{L^\infty(D)} = 1} \int_D b : a^{\frac 12} \frac{\partial X(a)}{\partial a} a^{\frac 12} \bigg)^2 \\ 
	&= \sum_D \bigg( \sup_{\| b\|_{L^\infty(D)} = 1} \limsup_{t\rightarrow 0} \frac{X \big(a + t a^{\frac 12} b a^{\frac 12} \big) - X(a)}{t} \bigg)^2.
	\end{align*}
	Next, we apply the elementary mean value theorem with some $\theta, \vartheta \in [0,1]$ to obtain 
	\begin{align*}
	&\frac{X \big(a + t a^{\frac 12} b a^{\frac 12} \big) - X(a)}{t} = \frac1t \fint_{B_R} \bigg( \big| a + t a^{\frac 12} b a^{\frac 12} \big|^p - | a |^p \bigg) \\ 
	&\qquad = p \fint_{B_R} \big( (1-\theta) |a| + \theta \big| a + t a^\frac12 b a^\frac12 \big| \big)^{p-1} \big| a + \vartheta t a^{\frac 12} b a^{\frac 12} \big|' : a^{\frac 12} b a^{\frac 12} \\ 
	&\qquad \lesssim \fint_{B_R} |a|^{p-1} \big( 1-\theta + \theta | I + t b | \big)^{p-1} \big| a^{\frac 12} b a^{\frac 12} \big| \\ 
	&\qquad \lesssim R^{-d} \int_{B_R \cap D} \mu^p,
	\end{align*}
	%where the first maximum is attained in $\ol \xi \in \partial B_1$, 
	where we use the identity $|A^2| = |A|^2$ for the spectral norm of any symmetric matrix $A \in \mathbb R^{d \times d}$ and the uniform boundedness $| |A|' : B | \lesssim |B|$ of the derivative of the spectral norm for any $A, B \in \mathbb R^{d \times d}$, while we assume w.l.o.g.\ that $\frac1t \big( X \big(a + t a^{\frac 12} b a^{\frac 12} \big) - X(a) \big)$ is positive for $t > 0$ sufficiently small. Note that we further employed the boundedness of $b$ and the fact that $b$ vanishes outside of $D$. Moreover, every instance of $a$ and $b$ inside an integral %as well as of $\xi$ and $\ol \xi$ 
	refers to $a(x)$ and $b(x)$, %$\xi(x, t)$, and $\ol \xi(x, t)$, 
	respectively. As a consequence, 
	\[
	\bigg\langle \Big( \sum_D \Big( \int_D \Big| \frac{\partial X(a)}{\partial a} \Big|_a \Big)^2 \Big)^r \bigg\rangle^\frac{1}{2r} \lesssim R^{-d} \bigg\langle \Big( \sum_{D \cap B_R \neq \emptyset} \Big( \int_{B_R \cap D} \mu^p \Big)^2 \Big)^r \bigg\rangle^\frac{1}{2r}.
	\]
	By recalling \eqref{eq:beta}, we notice that the number of subdomains $D$ obeying $D \cap B_R \neq \emptyset$ equals (up to fixed constants)
	\[
	\int_0^R \Big( \frac{l}{(l+1)^\beta} \Big)^{d-1} \frac{dl}{(l+1)^\beta} \lesssim \int_0^R l^{(1-\beta)(d-1) - \beta} dl \lesssim R^{d(1-\beta)}.
	\]
	Jensen's inequality then leads to 
	\[
	\bigg\langle \Big( \sum_D \Big( \int_D \Big| \frac{\partial X(a)}{\partial a} \Big|_a \Big)^2 \Big)^r \bigg\rangle^\frac{1}{2r} \lesssim R^{-\frac{d}{2} (1+\beta)} \bigg\langle R^{-d(1-\beta)} \sum_{D \cap B_R \neq \emptyset} \Big( \int_{D} \mu^p \Big)^{2r} \bigg\rangle^\frac{1}{2r}.
	\]
	Likewise, any subdomain $D$ can be covered by at most $|D| \lesssim R^{\beta d}$ unit balls $B_1(x_k)$ with appropriate $x_k \in D$, hence, applying Jensen's inequality once more results in 
	\[
	\bigg\langle \Big( \sum_D \Big( \int_D \Big| \frac{\partial X(a)}{\partial a} \Big|_a \Big)^2 \Big)^r \bigg\rangle^\frac{1}{2r} \lesssim R^{-\frac{d}{2} (1-\beta)} \bigg\langle R^{-d(1-\beta)} \sum_{D \cap B_R \neq \emptyset} R^{-\beta d} \sum_k \Big( \int_{B_1(x_k)} \mu^p \Big)^{2r} \bigg\rangle^\frac{1}{2r}.
	\]
	Pulling the expectation inside and using the stationarity of the underlying ensemble, we deduce 
	\[
	\bigg\langle \Big( \sum_D \Big( \int_D \Big| \frac{\partial X(a)}{\partial a} \Big|_a \Big)^2 \Big)^r \bigg\rangle^\frac{1}{2r} \lesssim R^{-\frac{d}{2} (1-\beta)} \bigg\langle \Big( \int_{B_1} \mu^p \Big)^{2r} \bigg\rangle^\frac{1}{2r}.
	\]
	Owing to \eqref{eq:strexpmom_mu} and Lemma \ref{lemma:expmom_polymom}, we know that 
	$
	\big\langle ( \int_{B_1} \mu^p )^{2r} \big\rangle^\frac{1}{2r} \lesssim r^\frac{1}{\alpha},
	$
	which gives rise to \eqref{eq:sgbound}. Together with \eqref{eq:2rmoment}, this results in 
	\[
	\Big\langle \big( X(a) - \langle X(a) \rangle \big)^{2r} \Big\rangle^\frac{1}{2r} \lesssim r^\frac{\alpha + 1}{\alpha} R^{-\frac{d}{2} (1-\beta)}.
	\]
	An elementary argument shows that 
	$
	\big\langle \big( X(a) - \langle X(a) \rangle \big)^{\frac{\alpha}{\alpha + 1} r} \big\rangle^\frac{1}{r} \lesssim r R^{-\frac{\alpha}{\alpha + 1} \frac{d}{2} (1-\beta)},
	$
	which by Lemma \ref{lemma:expmom_polymom} entails 
	\[
	\Big\langle \exp \Big( \frac{1}{C} R^{\frac{\alpha}{\alpha + 1} \frac{d}{2} (1-\beta)} \big( X(a) - \langle X(a) \rangle \big)^{\frac{\alpha}{\alpha + 1}} \Big) \Big\rangle < 2 
	\]
	for a sufficiently large constant $C > 0$ depending only on fixed model parameters. Therefore, 
	\begin{align}
	\label{eq:prob_delta}
	\Big\langle I \Big( \fint_{B_R} \mu^p - \langle \mu^p \rangle > \delta \Big) \Big\rangle \lesssim \exp \Big(\!-\frac{1}{C} \delta^{\frac{\alpha}{\alpha + 1}} R^{\frac{\alpha}{\alpha + 1} \frac{d}{2} (1-\beta)} \Big).
	\end{align}
	
	\medskip
	\emph{Step~2. Stretched exponential moments for $r_e$.}
	For any $r_0 > 1$ we now estimate the probability of the event $r_e > r_0$ as follows: 
	\begin{align*}
	\big\langle I(r_e > r_0) \big\rangle &\leq \Big\langle I \Big( \exists \, r > r_0 \ : \ \fint_{B_r} \mu^p > (2K)^p \ \lor \ \fint_{B_r} \lambda^{-q} > (2K)^q \Big) \Big\rangle \\ 
	&\leq \Big\langle I \Big( \exists \, n \geq n_0 \ : \ \fint_{B_{b^n}} \mu^p > \frac{(2K)^p}{b^d} \ \lor \ \fint_{B_{b^n}} \lambda^{-q} > \frac{(2K)^q}{b^d} \Big) \Big\rangle
	\end{align*}
	where $b = b(d, p, q, K) \in (1,2)$ is a constant specified below and $n_0, n \in \mathbb N$ satisfy $b^{n_0-1} < r_0 \leq b^{n_0}$ and $b^{n-1} < r \leq b^n$. Notice that $\fint_{B_{b^n}} \mu^p > (\frac{r}{b^n})^d (2K)^p > b^{-d} (2K)^p$ according to the assumption in the first line and that an analogous estimate holds for $\lambda^{-q}$. 
	%We introduce the random variable 
	%\[
	%X_L^S(A) \coleq \fint_{B_L} \big( A^S - \big\langle A^S \big\rangle \big) 
	%\]
	%for $L>0$, $S \geq 1$, and random fields $A$. 
	The previous estimate is continued via 
	\begin{align*}
	\big\langle I(r_e > r_0) \big\rangle \leq \sum_{n = n_0}^\infty \Big\langle I \Big( \fint_{B_{b^n}} \mu^p > \frac{(2K)^p}{b^d} \Big) \Big\rangle + \sum_{n = n_0}^\infty \Big\langle I \Big( \fint_{B_{b^n}} \lambda^{-q} > \frac{(2K)^q}{b^d} \Big) \Big\rangle.
	\end{align*}
	We are hence in a position to employ \eqref{eq:prob_delta} after choosing $b \in (1,2)$ sufficiently close to $1$ in order to guarantee that the lower bounds inside the indicator functions subsequently stay positive. Besides, we only provide the argument for the term involving $\mu$, while the same reasoning also applies to the other term. This yields 
	\begin{align*}
	&\sum_{n = n_0}^\infty \Big\langle I \Big( \fint_{B_{b^n}} \mu^p > \frac{(2K)^p}{b^d} \Big) \Big\rangle \leq \sum_{n = n_0}^\infty \Big\langle I \Big( \fint_{B_{b^n}} \mu^p - \langle \mu^p \rangle > \frac{(2K)^p}{b^d} - K^p \Big) \Big\rangle \\ 
	&\qquad \leq \sum_{n = n_0}^\infty \exp \Big( -\frac{1}{\ol C} \Big( \frac{(2K)^p}{b^d} - K^p \Big)^\frac{\alpha}{\alpha + 1} \big( b^n \big)^{\frac{\alpha}{\alpha + 1} \frac{d}{2} (1-\beta)} \Big) \\ 
	&\qquad \leq C_0 \exp \Big(\!-c_1 r_0^{\frac{\alpha}{\alpha + 1} \frac{d}{2} (1-\beta)} \Big)
	\end{align*}
	together with (large) constants $\ol C, C_0 \geq 1$, and a (small) constant $c_1 \in (0, 1)$.   
	%\[
	%\pi_*(\ell) \simeq 
	%\begin{cases}
	%(\ell + 1)^\alpha, &\alpha < d, \\ 
	%(\ell + 1)^d \log^{-1}(\ell + 2), &\alpha = d, \\ 
	%(\ell + 1)^d, &\alpha > d,
	%\end{cases}
	%\qquad \text{and} \qquad 
	%\alpha_* \coleq 
	%\begin{cases}
	%\alpha, &\alpha < d, \\ 
	%d - \epsilon, &\alpha = d, \\ 
	%d, &\alpha > d.
	%\end{cases}
	%\]
	For the third inequality above, we pull out the factor $\exp \big( -c_1 (b^{n_0})^{\frac{\alpha}{\alpha + 1} \frac{d}{2} (1-\beta)} \big)$ and estimate the remaining sum using the crude bound $(b^n)^{\frac{\alpha}{\alpha + 1} \frac{d}{2} (1-\beta)} - (b^{n_0})^{\frac{\alpha}{\alpha + 1} \frac{d}{2} (1-\beta)} \gtrsim \log(b) (n - n_0)$. 
	We can now derive moment bounds of order $k \geq 1$ via 
	\begin{align*}
	\big\langle r_e^k \big\rangle &= \Big\langle \int_0^\infty I(r_e > r) \frac{\d}{\d r} r^k \d r \Big\rangle \leq \int_0^\infty C_0 \exp \Big(\!-c_1 r^{\frac{\alpha}{\alpha + 1} \frac{d}{2} (1-\beta)} \Big) \frac{\d}{\d r} r^k \d r \\ 
	&= \int_0^\infty C_0 c_1 {\frac{\alpha}{\alpha + 1} \frac{d}{2} (1-\beta)} r^{\frac{\alpha}{\alpha + 1} \frac{d}{2} (1-\beta) - 1} \exp \Big(\!-c_1 r^{\frac{\alpha}{\alpha + 1} \frac{d}{2} (1-\beta)} \Big) r^k \d r. 
	\end{align*}
	By means of the substitution $t = c_1 r^{\frac{\alpha}{\alpha + 1} \frac{d}{2} (1-\beta)}$, the last expression rewrites as 
	\begin{align*}
	\big\langle r_e^k \big\rangle &= \int_0^\infty C_0 e^{-t} \Big( \frac{t}{c_1} \Big)^{\frac{\alpha + 1}{\alpha} \frac{2k}{d(1-\beta)}} \d t = C_0 c_1^{-\frac{\alpha + 1}{\alpha} \frac{2k}{d(1-\beta)}} \Gamma \Big( \frac{\alpha + 1}{\alpha} \frac{2k}{d(1-\beta)} + 1 \Big).
	\end{align*}
	The stretched exponential bound \eqref{eq:strexpmom_re} for $r_e$ of order $\frac{\alpha}{\alpha + 1} \frac{d}{2} (1-\beta)$ now immediately follows as 
	\begin{align}
	\label{eq:strexpmom_prooffinal}
	\Big\langle \exp \Big( \frac1C r_e^{\frac{\alpha}{\alpha + 1} \frac{d}{2} (1-\beta)} \Big) \Big\rangle \leq 1 + \sum_{k=1}^\infty \frac{\big\langle r_e^{\frac{\alpha}{\alpha + 1} \frac{d}{2} (1-\beta) k} \big\rangle}{C^k k!} \leq 1 + C_0 \sum_{k=1}^\infty \frac{\Gamma(k+1)}{C^k c_1^k k!} < 2
	\end{align}
	for $C > 0$ large enough. %In case that the stronger hypothesis \eqref{eq:alsurebound_mu} is imposed, \eqref{eq:sgbound} obviously holds without the factor $r^\frac{1}{\alpha}$. Consequently, \eqref{eq:prob_delta} and in particular \eqref{eq:strexpmom_prooffinal} hold with the factor $\frac{\alpha}{\alpha + 1}$ being replaced by $1$, thereby showing \eqref{eq:alsurebound_re}.
\end{proof}

\subsection{Proof of Proposition \ref{propsensitivity}: A sensitivity estimate for average integrals}

\begin{remark}
	\label{rem:a12}
	The reason for demanding symmetric coefficient fields $a(x) = a(x)^T$ is mainly related to the subsequent proof of Proposition \ref{propsensitivity}, which does not seem to generalize to the case of non-symmetric $a$. In particular, the arguments in \eqref{eq:ainv_needed} and \eqref{eq:a12_needed} heavily rely on the symmetry of $a$. In \eqref{eq:a12_needed}, we smuggle in $a^{-\frac 12}$ and $a^\frac 12$ leading to $|g|_{a^{-1}}$ and $|\psi|_a$ after applying H\"older's inequality. 
	In the absence of symmetry and assuming that an appropriate notion of the square root of a matrix is chosen, 
	the terms which we insert should still cancel and be of the order $-\frac 12$ and $\frac 12$ w.r.t.\ $a$. Owing to \eqref{eq:ainv_needed} and the fact that the matrix $a$ should partially cancel within the norm $|\cdot|_{a^{-1}}$, we see that the definition of $|\cdot|_{a^{-1}}$ has to be of the form $|g|_{a^{-1}}^2 \coleq g^T (a^\frac 12)^{-T} (a^\frac 12)^{-1} g$. Since \eqref{eq:ainv_needed} shall be controlled in terms of $|\cdot|_a$ in the subsequent estimate, $|\cdot|_a$ needs to be defined as $|\psi|_a^2 \coleq \psi^T (a^\frac12)^T a^\frac 12 \psi$. But the structure of $|\cdot|_{a^{-1}}$ and $|\cdot|_a$ now prevents us from proceeding as in \eqref{eq:a12_needed} unless $a$ is symmetric. 
\end{remark}

\begin{proof}[Proof of Proposition \ref{propsensitivity}]
We follow a strategy similar to the one in \cite{GNO20v2},\, which separates the proof into several steps. Throughout the proof, we will use the notation 
\[
F \nabla (\phi, \sigma) \coleq \int \tilde g \cdot \nabla \phi + \int \bar g \cdot \nabla \sigma
\]
for compactly supported $g = (\tilde g, \bar g)$.

\medskip
\emph{Step~1. Energy estimate for all $r \geq r_e$.} 
We claim that any $a$-harmonic function $u$ on $\mathbb R^d$, i.e.\ any solution to 
\begin{align}
\label{eqaharm}
-\nabla \cdot a \nabla u = 0,
\end{align}
satisfies 
\begin{align}
\label{eqholefill}
\int_{B_r(x)} |\nabla u|_a^2 \leq 4^d \Big( \frac r R \Big)^{\varepsilon d} \int_{B_R(x)} |\nabla u|_a^2
\end{align}
for some $\varepsilon > 0$ and for all $R \geq r \geq r_e$, where generic constants here and in the subsequent proof only depend on the dimension $d$. 

%We begin with the Caccioppoli estimate 
%\[
%\int_{B_{r_e}(x)} |\nabla u|_a^2 \lesssim r_e^{-2} \int_{B_{2r_e}(x) \backslash B_{r_e}(x)} a ( u - c )^2
%\]
%where $c>0$ will be specified below and which easily follows from testing \eqref{eqaharm} with $\eta^2 (u-c)$. We further use $\eta$ as a cutoff function for $B_{r_e}(x)$ in $B_{2r_e}(x)$. We now aim to apply a weighted Poincar\'e inequality in the form 
%\[
%\int_{B_{r_e}(x)} |\nabla u|_a^2 \lesssim r_e^{-2} \int_{B_{2r_e}(x) \backslash B_{r_e}(x)} a ( u - c )^2
%\]
%\[
%c \coleq \int_{B_{2r_e}(x) \backslash B_{r_e}(x)} \frac{a}{\int_{B_{2r_e}(x) \backslash B_{r_e}(x)} a} \, u
%\]

Applying the Caccioppoli estimate for solutions to \eqref{eqaharm} from \cite[Lemma 3]{BFO18} entails 
\[
\fint_{B_{r_e}(x)} |\nabla u|_a^2 \lesssim K r_e^{-2} \bigg( \fint_{A} | u - \bar u |^\frac{2p}{p-1} \bigg)^\frac{p-1}{p}
\]
where we abbreviate $A \coleq B_{2r_e}(x) \backslash B_{r_e}(x)$ and $\bar u \coleq \int_{A} u$. Since the condition $\frac{1}{p} + \frac{1}{q} \leq \frac{2}{d}$ guarantees the embedding $W^{1,\frac{2q}{q+1}}(A) \hookrightarrow L^\frac{2p}{p-1}(A)$, we infer from Sobolev's inequality (observing the correct scaling w.r.t.\ $r_e$) that 
\begin{align*}
\fint_{B_{r_e}(x)} |\nabla u|_a^2 &\lesssim K r_e^{-2-d\frac{p-1}{p}} \bigg( \int_{A} | u - \bar u |^\frac{2p}{p-1} \bigg)^\frac{p-1}{p} \\ 
&\lesssim C_\mathrm{Sob}^2 K r_e^{-2-d\frac{p-1}{p}} r_e^{-d(\frac1p + \frac1q)} \bigg( \Big( \int_{A} | u - \bar u |^\frac{2q}{q+1} \Big)^\frac{q+1}{q} +r_e^2 \Big( \int_{A} | \nabla u |^\frac{2q}{q+1} \Big)^\frac{q+1}{q} \bigg) \nonumber \\
&\lesssim C_\mathrm{Sob}^2 C_\mathrm{Poi}^2 K r_e^{-2-d\frac{p-1}{p}} r_e^{-d(\frac1p + \frac1q)} r_e^2 \Big( \int_{A} | \nabla u |^\frac{2q}{q+1} \Big)^\frac{q+1}{q} \nonumber \\
&= C_\mathrm{Sob}^2 C_\mathrm{Poi}^2 K r_e^{-d\frac{q+1}{q}} \Big( \int_{A} | \nabla u |^\frac{2q}{q+1} \Big)^\frac{q+1}{q} \nonumber
\end{align*}
where we employed the correctly scaled Poincar\'e inequality for the third estimate. Thanks to H\"older's inequality (cf.\ \cite[Lemma 3]{BFO18}), we conclude that 
\begin{align*}
\fint_{B_{r_e}(x)} |\nabla u|_a^2 \lesssim C_\mathrm{Sob}^2 C_\mathrm{Poi}^2 K \Big( \fint_{A} | \nabla u |^\frac{2q}{q+1} \Big)^\frac{q+1}{q} \lesssim C_\mathrm{Sob}^2 C_\mathrm{Poi}^2 K^2 \fint_{A} | \nabla u |_a^2.
\end{align*}
Rewriting this estimate in terms of an explicit constant $E \coleq C_d C_\mathrm{Sob}^2 C_\mathrm{Poi}^2 K^2$, we get
\[
\int_{B_{r_e}(x)} |\nabla u|_a^2 \leq E \int_{B_{2r_e}(x) \backslash B_{r_e}(x)} |\nabla u|_a^2.
\] 
The bound in \eqref{eqholefill} now follows from a so-called hole-filling trick, which amounts to adding $E \int_{B_{r_e}(x)} |\nabla u|_a^2$ to both sides. This results in 
\[
\int_{B_{r_e}(x)} |\nabla u|_a^2 \leq \frac{E}{E+1} \int_{B_{2r_e}(x)} |\nabla u|_a^2 \leq \Big( \frac{E}{E+1} \Big)^n \int_{B_{2^n r_e}(x)} |\nabla u|_a^2
\]
where the second bound simply follows by iteration. Defining $\varepsilon > 0$, $n \in \mathbb N$, and $N \in \mathbb N$ via 
\begin{align}
\label{eq:defholefill}
\frac{E}{E+1} = 2^{-\varepsilon d}, \qquad 2^{n-1} r_e \leq r < 2^n r_e, \qquad 2^{N-1} r_e \leq R < 2^N r_e,
\end{align}
we arrive at 
\[
\int_{B_r(x)} |\nabla u|_a^2 \leq \int_{B_{2^n r_e}(x)} |\nabla u|_a^2 \leq 2^{-\varepsilon d (N-1-n)} \int_{B_{2^{N-1} r_e}(x)} |\nabla u|_a^2 \leq 2^{2\varepsilon d} \Big( \frac r R \Big)^{\varepsilon d} \int_{B_R(x)} |\nabla u|_a^2
\]
provided $N-1 \geq n$. 
In case that $N = n$, this inequality obviously holds true. To ensure that the expression on the right-hand side (or an upper bound thereof) is increasing for decreasing $\varepsilon$, we simply skip $\varepsilon \in (0,1)$ within the factor $2^{2\varepsilon d}$. 

\medskip
\emph{Step~2. Energy estimates for $r \geq r_*$.} We start by noting that 
\begin{align}
\label{eqcaccioppoli_large}
\fint_{B_R} | \nabla \phi_i + e_i |_a^2 \lesssim 1
\end{align}
for all $R \geq r$, where the constant on the right-hand side only depends on $d$ and $K$. This follows from a Caccioppoli estimate as stated in \cite[Lemma 3]{BFO18} since $r_* \geq r_e$ in particular ensures $R \geq r_e$. The definition of the minimal radius $r_*$ in Definition \ref{defminradius} then allows for a constant upper bound. 

As in \cite{GNO20v2}, we now claim that for all $\gamma \in (0,d)$ and any decaying functions $u$ and $g$ related via 
\begin{align}
\label{eqdivg}
-\nabla \cdot a \nabla u = \nabla \cdot g,
\end{align}
we have 
\begin{align}
\label{eqoptimalenergygamma}
\int_{B_r} |\nabla u|_a^2 \lesssim \int \Big( \frac{|x|}{r} + 1 \Big)^{-\gamma} |g|_{a^{-1}}^2
\end{align}
with generic constants only depending on $d$ and $\gamma$ in this paragraph. 
Restricting oneself by scaling to the case $r=1$, one is left to prove 
\[
\bigg( \int_{B_1} |\nabla u|_a^2 \bigg)^\frac12 \lesssim \bigg( \int_{B_1} |g|_{a^{-1}}^2 \bigg)^\frac12 + \sum_{n=1}^\infty \bigg( \frac{1}{(2^n)^d} \int_{2^{n-1}<|x|<2^n} |g|_{a^{-1}}^2 \bigg)^\frac12
\]
taking the following elementary estimate into account: 
\begin{align*}
&\bigg( \int_{B_1} |g|_{a^{-1}}^2 \bigg)^\frac12 + \sum_{n=1}^\infty \bigg( \frac{1}{(2^n)^d} \int_{2^{n-1}<|x|<2^n} |g|_{a^{-1}}^2 \bigg)^\frac12 \\ 
&\qquad \lesssim \bigg( \int_{B_1} \big( |x| + 1 \big)^{-\gamma} |g|_{a^{-1}}^2 \bigg)^\frac12 + \sum_{n=1}^\infty \big(2^n\big)^\frac{\gamma-d}{2} \bigg( \int_{2^{n-1}<|x|<2^n} \big( |x| + 1 \big)^{-\gamma} |g|_{a^{-1}}^2 \bigg)^\frac12 \\ 
&\qquad \lesssim \bigg( \int \big( |x| + 1 \big)^{-\gamma} |g|_{a^{-1}}^2 \bigg)^\frac12.
\end{align*}
As a result of the unique solvability of \eqref{eqdivg} in the class of decaying solutions, we may assume that $g$ is either supported in $B_1$ or in $B_{2^n} \backslash B_{2^{n-1}}$ for some $n \in \mathbb N$. In the first case, we employ the energy estimate for \eqref{eqdivg} to derive $\int_{B_1} |\nabla u|_a^2 \leq \int |\nabla u|_a^2 \lesssim \int_{B_1} |g|_{a^{-1}}^2$. If $\supp g \subset B_{2^n} \backslash B_{2^{n-1}}$, we additionally use the mean-value property from Lemma \ref{lemmac1alpha} to deduce $\int_{B_1} |\nabla u|_a^2 \lesssim (2^n)^{-d} \int_{|x| < 2^n} |\nabla u|_a^2 \lesssim (2^{n})^{-d} \int_{2^{n-1}<|x|<2^n} |g|_{a^{-1}}^2$.

In order to ensure an appropriate bound on the gradient of the flux correction, $\nabla \sigma$, we also need the subsequent result. For all $\gamma \in (0, d)$ and decaying functions $u$ and $g$ satisfying 
\[
-\Delta u = \nabla \cdot g,
\]
we have 
\begin{align}
\label{eqoptimalenergy-caldzygm}
\bigg( \int_{B_r} | \nabla u |^\frac{2p}{p+1} \bigg)^\frac{p+1}{p} \lesssim \int \Big( \frac{|x|}{r} + 1 \Big)^{-\gamma} |g|_{a^{-1}}^2, 
\end{align}
where generic constants may depend on $d$, $\gamma$, and $K$ in this context. As above, it is enough to prove the result for $r=1$ and under the additional assumption that $\supp g \subset B_2$ or $\supp g \subset B_{2^n} \backslash B_{2^{n-1}}$ for some $n \in \mathbb N$, $n \geq 2$. In the former case, we invoke a Calder\'on--Zygmund estimate (see, e.g., \cite[Subsections 7.1.2--7.1.3]{GM12}) to find 
\begin{align*}
\int_{B_1} | \nabla u |^\frac{2p}{p+1} \lesssim \int_{B_2} | g |^\frac{2p}{p+1} \lesssim \int_{B_2} \big| a^\frac12 \big|^\frac{2p}{p+1} \big| a^{-\frac12} g \big|^\frac{2p}{p+1} \lesssim \bigg( \int_{B_2} |a|^p \bigg)^\frac{1}{p+1} \bigg( \int_{B_2} |g|_{a^{-1}}^2 \bigg)^\frac{p}{p+1}. 
\end{align*}
As $r \geq r_\ast \geq r_e$, and since we work with the scaling $r = 1$, we conclude that 
\begin{align}
\label{eq:caldzyg_1}
\bigg( \int_{B_1} | \nabla u |^\frac{2p}{p+1} \bigg)^\frac{p+1}{p} %\lesssim \bigg( \int_{B_2} |a|^p \bigg)^\frac{1}{p} \int_{B_2} |g|_{a^{-1}}^2 
\lesssim \int_{B_2} |g|_{a^{-1}}^2.
\end{align}
In the latter case, we first recall that the mean-value property also holds true for the $L^\frac{2p}{p+1}$ norm. This follows from Jensen's inequality and the standard mean-value property of harmonic functions, namely $u = \frac{1}{|B_R|} \mathbb{1}_{B_R} \ast u$ and, hence, $\nabla u(x) = \fint_{B_R(x)} \nabla u(y) \, dy$ for $R > 0$: Specifying $R \coleq 2^{n-1} - 1$, we derive $|\nabla u(x)|^\frac{2p}{p+1} \leq \fint_{B_R(x)} |\nabla u(y)|^\frac{2p}{p+1} \, dy \leq \big( \frac{2^n}{R} \big)^d \fint_{B_{2^n}} |\nabla u(y)|^\frac{2p}{p+1} \, dy \leq 4^d \fint_{B_{2^n}} |\nabla u(y)|^\frac{2p}{p+1} \, dy$. As a consequence, %from Lemma \ref{lemmac1alpha} with $a \coleq \mathrm{id}$, 
%and from a Caccioppoli estimate followed by Sobolev's and Poincar\'e's inequality similar to \eqref{eq:cacc-sob-poinc}:
\begin{align*}
\int_{B_1} | \nabla u |^\frac{2p}{p+1} \lesssim \fint_{B_{2^n}} | \nabla u |^\frac{2p}{p+1}.
\end{align*}
Applying now an analogue of \eqref{eq:caldzyg_1} on $B_{2^n}$, we infer that 
\begin{align}
\label{eq:caldzyg_2}
\bigg( \int_{B_1} | \nabla u |^\frac{2p}{p+1} \bigg)^\frac{p+1}{p} &\lesssim (2^n)^{-d \frac{p+1}{p}} \bigg( \int_{|x| < 2^n} | \nabla u |^\frac{2p}{p+1} \bigg)^\frac{p+1}{p} \\ 
&\lesssim (2^n)^{-d} \int_{2^{n-1} < |x| < 2^n} |g|_{a^{-1}}^2. \nonumber
\end{align}
By the same arguments as above, we see that \eqref{eq:caldzyg_1} and \eqref{eq:caldzyg_2} give rise to \eqref{eqoptimalenergy-caldzygm}.

The generalization of the previous estimate \eqref{eqoptimalenergygamma} provided in \cite{GNO20v2} also holds in our situation. For any $0 < \gamma' < \gamma < d$ and functions $u$ and $g$ subject to \eqref{eqdivg}, one has 
\begin{align}
\label{eqoptimalenergygamma'}
\int \Big( \frac{|x|}{r} + 1 \Big)^{-\gamma} |\nabla u|_a^2 \lesssim \int \Big( \frac{|x|}{r} + 1 \Big)^{-\gamma'} |g|_{a^{-1}}^2,
\end{align}
where here the generic constants only depend on $d$, $\gamma$, and $\gamma'$. As above, we restrict ourselves to $r=1$. By using \eqref{eqoptimalenergygamma} with $r \coleq 2\rho$ therein for some $\rho \geq 1$, we find 
\[
\int_{\rho < |x| < 2\rho} |\nabla u|_a^2 \lesssim \int \Big( \frac{|x|}{2 \rho} + 1 \Big)^{-\gamma'} |g|_{a^{-1}}^2 \lesssim \rho^{\gamma'} \int \big( |x| + 1 \big)^{-\gamma'} |g|_{a^{-1}}^2. 
\]
Multiplying with $\rho^{-\gamma}$, we obtain $\int_{\rho < |x| < 2\rho} |x|^{-\gamma} |\nabla u|_a^2 \lesssim \rho^{\gamma'-\gamma} \int ( |x| + 1 )^{-\gamma'} |g|_{a^{-1}}^2$. Setting $\rho \coleq 2^n$, recalling $\gamma' < \gamma$, and taking the sum over $n \in \mathbb N$, we arrive at
\[
\int_{|x| \geq 1} \big( |x| + 1 \big)^{-\gamma} |\nabla u|_a^2 \lesssim \int \big( |x| + 1 \big)^{-\gamma'} |g|_{a^{-1}}^2.
\]
The case $|x| < 1$ is treated by the previous estimate \eqref{eqoptimalenergygamma} for $r=1$: 
\[
\int_{|x| < 1} \big( |x| + 1 \big)^{-\gamma} |\nabla u|_a^2 \leq \int_{|x| < 1} |\nabla u|_a^2 \lesssim \int \big( |x| + 1 \big)^{-\gamma'} |g|_{a^{-1}}^2.
\]

\medskip
\emph{Step~3. Sensitivity estimate for all $r \geq r_*$.} We proceed by following \cite{GNO20v2} and recall that the defining equations for the decaying functions $\phi$ and $\sigma_{jk}$ (where we skip the index $i$ for notational convenience) read 
\begin{align}
\label{eqphisigma_sensitivity}
-\nabla \cdot a ( \nabla \phi + e ) = 0, \qquad -\Delta \sigma_{jk} = \partial_j q_k - \partial_k q_j, \qquad q = a(\nabla\phi + e). % - a_\mathrm{hom} e.
\end{align}
Fixing an element $D$ of the underlying partition of $\mathbb R^d$, we shall write $a_D$ for a coefficient field which may differ from $a$ only inside of $D$. The corresponding solutions to \eqref{eqphisigma_sensitivity} for $a$ replaced by $a_D$ are then denoted by $\phi_D$ and $\sigma_{jkD}$. As a result, the differences $\phi - \phi_D$ and $\sigma - \sigma_{jkD}$ are subject to 
\begin{align*}
-\nabla \cdot a \nabla (\phi - \phi_D) &= \nabla \cdot (a - a_D) (\nabla \phi_D + e), \\ 
-\Delta (\sigma_{jk} - \sigma_{jkD}) &= \partial_j \big( a (\nabla \phi + e) - a_D (\nabla \phi_D + e) \big)_k - \partial_k \big( a (\nabla \phi + e) - a_D (\nabla \phi_D + e) \big)_j.
\end{align*}
Taking a linear combination with scalar coefficients $\{c_D\}_D$, we get 
\begin{align*}
-\nabla \cdot a \nabla \sum_D c_D (\phi - \phi_D) &= \nabla \cdot \sum_D c_D (a - a_D) (\nabla \phi_D + e), \\ 
-\Delta \sum_D c_D (\sigma_{jk} - \sigma_{jkD}) &= \partial_j \sum_D c_D \big( a (\nabla \phi + e) - a_D (\nabla \phi_D + e) \big)_k \\ 
&\quad - \partial_k \sum_D c_D \big( a (\nabla \phi + e) - a_D (\nabla \phi_D + e) \big)_j.
\end{align*}
With the help of estimate \eqref{eqoptimalenergygamma'} in Step 2, we now derive 
\begin{align*}
&\int_{B_r} \Big| \nabla \sum_D c_D (\phi - \phi_D) \Big|_a^2 + \int \Big( \frac{|x|}{r} + 1 \Big)^{-\gamma} \Big| \nabla \sum_D c_D (\phi - \phi_D) \Big|_a^2 \\ 
&\qquad \lesssim \int \Big( \frac{|x|}{r} + 1 \Big)^{-\gamma'} \Big| \sum_D c_D (a - a_D) (\nabla \phi_D + e) \Big|_{a^{-1}}^2 \\ 
&\qquad \leq \int \Big( \frac{|x|}{r} + 1 \Big)^{-\gamma'} \sum_D c_D^2 \Big| a^{-\frac 12} (a - a_D) a^{-\frac 12} a^{\frac 12} (\nabla \phi_D + e) \Big|^2 \\ 
&\qquad \leq \sum_D c_D^2 \sup_{x \in D} \big| a - a_D \big|_{a^{-1}}^2 \int_D \Big( \frac{|x|}{r} + 1 \Big)^{-\gamma'}  | \nabla \phi_D + e |_a^2.
\end{align*}
%We refer to Notation \ref{not:as} for the definition of $a^\frac S2$ and $a^{-\frac S2}$. 
Note that we crucially employed the inclusion $\supp (a - a_D) \subset D$ and the fact that the elements of the partition are disjoint when pulling the square inside the sum over all $D$. Moreover, all generic constants appearing in this part of the proof only depend on $d$, $\gamma$, and $\gamma'$. Next, we observe that 
\[
a (\nabla \phi + e) - a_D (\nabla \phi_D + e) = a \nabla (\phi - \phi_D) + (a - a_D) (\nabla \phi_D + e)
\]
together with the previous estimates and $\gamma' \leq \gamma$ leads to 
\begin{align}
&\int_{B_r} \Big| \sum_D c_D (\nabla \phi - \nabla \phi_D) \Big|_a^2 + \int \Big( \frac{|x|}{r} + 1 \Big)^{-\gamma} \Big| \sum_D c_D \big( a (\nabla \phi + e) - a_D (\nabla \phi_D + e) \big) \Big|_{a^{-1}}^2 \nonumber \\ 
&\qquad \leq \int_{B_r} \Big| \sum_D c_D (\nabla \phi - \nabla \phi_D) \Big|_a^2 + \int \Big( \frac{|x|}{r} + 1 \Big)^{-\gamma} \Big| a \nabla \sum_D c_D (\phi - \phi_D) \Big|_{a^{-1}}^2 \label{eq:ainv_needed} \\
&\qquad\qquad + \int \Big( \frac{|x|}{r} + 1 \Big)^{-\gamma'} \Big| \sum_D c_D (a - a_D) (\nabla \phi_D + e) \Big|_{a^{-1}}^2 \nonumber \\ 
&\qquad \lesssim \sum_D c_D^2 \sup_{x \in D} \big| a - a_D \big|_{a^{-1}}^2 \int_D \Big( \frac{|x|}{r} + 1 \Big)^{-\gamma'}  | \nabla \phi_D + e |_a^2. \nonumber
\end{align}
Similarly, we apply \eqref{eqoptimalenergy-caldzygm} to find 
\begin{align*}
\bigg( \int_{B_r} \Big| \nabla \sum_D c_D ( \sigma_{jk} - \sigma_{jkD} ) \Big|^\frac{2p}{p+1} \bigg)^\frac{p+1}{p} %\lesssim \int \Big( \frac{|x|}{r} + 1 \Big)^{-\gamma} \Big| \sum_D c_D \big( a (\nabla \phi + e) - a_D (\nabla \phi_D + e) \big) \Big|^2 \\ 
&\lesssim \int \Big( \frac{|x|}{r} + 1 \Big)^{-\gamma} \Big| \sum_D c_D \big( a (\nabla \phi + e) - a_D (\nabla \phi_D + e) \big) \Big|_{a^{-1}}^2 \\ 
&\lesssim \, \sum_D c_D^2 \sup_{x \in D} \big| a - a_D \big|_{a^{-1}}^2 \int_D \Big( \frac{|x|}{r} + 1 \Big)^{-\gamma'}  | \nabla \phi_D + e |_a^2. 
\end{align*}
Due to the definition of the functional 
\[
F\psi = \int g \cdot \psi \quad \text{with} \quad \Big( \fint_{B_r} |g|_{a^{-1}}^2 \Big)^\frac12 \lesssim K^\frac12 \Big( \fint_{B_r} |g|^\frac{2q}{q-1} \Big)^\frac{q-1}{2q} \lesssim r^{-d},
\]
we have 
\begin{align}
\label{eq:a12_needed}
|F\psi| = \Big| \int_{B_r} g^T a^{-\frac 12} a^\frac 12 \psi \Big| \leq \Big( \int_{B_r} |a^{-\frac12} g|^2 \Big)^\frac12 \Big( \int_{B_r} |a^{\frac12} \psi|^2 \Big)^\frac12 \leq \Big( \fint_{B_r} |\psi|_a^2 \Big)^\frac12
\end{align}
and 
\begin{align*}
|F\psi| \leq \Big( \int_{B_r} |g|^\frac{2p}{p-1} \Big)^\frac{p-1}{2p} \Big( \int_{B_r} |\psi|^\frac{2p}{p+1} \Big)^\frac{p+1}{2p} \lesssim r^{-d \frac{p+1}{2p}} \Big( \int_{B_r} |\psi|^\frac{2p}{p+1} \Big)^\frac{p+1}{2p} \lesssim \Big( \fint_{B_r} |\psi|^\frac{2p}{p+1} \Big)^\frac{p+1}{2p}.
\end{align*}
Consequently, we deduce that 
\begin{align*}
&r^d \Big| \sum_D c_D ( F\nabla \phi - F\nabla \phi_D ) \Big|^2 + r^d \Big| \sum_D c_D ( F\nabla \sigma_{jk} - F\nabla \sigma_{jkD} ) \Big|^2 \\ 
&\qquad \lesssim \sum_D c_D^2 \sup_{x \in D} \big| a - a_D \big|_{a^{-1}}^2 \int_D \Big( \frac{|x|}{r} + 1 \Big)^{-\gamma'}  | \nabla \phi_D + e |_a^2.
\end{align*}
Thus, an elementary $l^2$ duality argument guarantees the bound 
\begin{align*}
&r^d \sum_D | F\nabla \phi - F\nabla \phi_D |^2 + r^d \sum_D | F\nabla \sigma_{jk} - F\nabla \sigma_{jkD} |^2 \\ 
&\qquad \lesssim \sup_D \bigg( \sup_{x \in D} \big| a - a_D \big|_{a^{-1}}^2 \int_D \Big( \frac{|x|}{r} + 1 \Big)^{-\gamma'}  | \nabla \phi_D + e |_a^2 \bigg).
\end{align*}

We now specify $a_D \coleq a + t a^{\frac12} \delta a_D a^{\frac12}$ for $t \in \mathbb R$, $|t| \ll 1$, and a perturbation $\delta a_D$ being bounded and supported in $D$. In the limit $t \rightarrow 0$, the previous estimate becomes 
\begin{align*}
&r^d \sum_D \Big| \int_D \frac{\partial F \nabla \phi}{\partial a} : a^{\frac12} \delta a_D a^{\frac12} \Big|^2 + r^d \sum_D \Big| \int_D \frac{\partial F \nabla \sigma_{jk}}{\partial a} : a^{\frac12} \delta a_D a^{\frac12} \Big|^2 \\ 
&\qquad \lesssim \sup_D \sup_{x \in D} | \delta a_D |^2 \cdot \sup_D \int_D \Big( \frac{|x|}{r} + 1 \Big)^{-\gamma'}  | \nabla \phi_D + e |_a^2.
\end{align*}
Lemma \ref{lemmamatrixduality} now allows for an explicit estimate of the matrix-valued derivative of $F$ in terms of specific matrix norms. But as all matrix norms on $\mathbb R^{d \times d}$ are equivalent, Lemma \ref{lemmamatrixduality} also holds for the spectral norm $|\cdot|$ up to an additional constant. Hence, we arrive at 
\begin{align}
\label{eq:sensitivity_Fderiv}
&r^d \sum_D \Big( \int_D \Big| \frac{\partial F \nabla \phi}{\partial a} \Big|_a \Big)^2 + r^d \sum_D \Big( \int_D \Big| \frac{\partial F \nabla \sigma_{jk}}{\partial a} \Big|_a \Big)^2 \\ 
&\qquad \lesssim \sup_D \int_D \Big( \frac{|x|}{r} + 1 \Big)^{-\gamma'} | \nabla \phi_D + e |_a^2 \lesssim \sup_D \Big( \frac{\dist D}{r} + 1 \Big)^{-\gamma'} \int_D | \nabla \phi_D + e |_a^2. \nonumber
\end{align}

To derive the announced sensitivity estimate for $r \geq r_*$, we set $\rho \coleq \diam D$ and choose $x \in D$ such that $R \coleq \dist D = |x|$. The assumption $\rho \leq (R+1)^\beta$ with $\beta \in (0,1)$ on the coarseness of the partition as well as the hole-filling estimate \eqref{eqholefill} ensure that 
\[
\int_D | \nabla \phi_D + e |_a^2 \leq \int_{B_\rho(x)} | \nabla \phi_D + e |_a^2 \lesssim \Big( \frac{1}{R+1} \Big)^{\varepsilon d (1-\beta)} \int_{B_{R+1}(x)} | \nabla \phi_D + e |_a^2.
\]
%Enlarging the domain of integration and 
Applying \eqref{eqholefill} once more in case that $r_* > 2R+1$, we obtain  
\[
\int_D | \nabla \phi_D + e |_a^2 %\lesssim \Big( \frac{1}{R+1} \Big)^{\varepsilon d (1-\beta)} \Big( \frac{2R+1}{r_*} \Big)^{\varepsilon d} \int_{B_{\max\{2R+1,r_*\}}} | \nabla \phi_D + e |_a^2. %
\lesssim \Big( \frac{1}{R+r_*+1} \Big)^{\varepsilon d (1-\beta)} \int_{B_{\max\{2R+1,r_*\}}} | \nabla \phi_D + e |_a^2.
\]
While this estimate trivially follows from the previous one in case that $r_* \leq 2R+1$, we treat the additional factor from the hole-filling estimate in case that $r_* > 2R+1$ via $\big( \frac{1}{R+1} \big)^{\varepsilon d (1-\beta)} \big( \frac{2R+1}{r_*} \big)^{\varepsilon d} \leq \big( \frac{1}{R+1} \big)^{\varepsilon d (1-\beta)} \big( \frac{2R+1}{r_*} \big)^{\varepsilon d (1-\beta)} \lesssim \big( \frac{1}{r_*} \big)^{\varepsilon d (1-\beta)} \lesssim \big( \frac{1}{R+r_*+1} \big)^{\varepsilon d (1-\beta)}$. The remaining integral is controlled via the Caccioppoli estimate \eqref{eqcaccioppoli_large}, which results in 
\[
\int_D | \nabla \phi_D + e |_a^2 \lesssim \Big( \frac{1}{R+r_*+1} \Big)^{\varepsilon d (1-\beta)} (R+r_*+1)^d \leq (R+r)^{d(1-\varepsilon(1-\beta))}
\]
recalling $r \geq r_* \geq 1$. Going back to \eqref{eq:sensitivity_Fderiv} and defining $\gamma' \coleq d(1-\varepsilon(1-\beta)) < d$, we conclude that 
\begin{align}
\label{eq:sensitivity_largescale}
r^d \sum_D \Big( \int_D \Big| \frac{\partial F \nabla \phi}{\partial a} \Big|_a \Big)^2 + r^d \sum_D \Big( \int_D \Big| \frac{\partial F \nabla \sigma_{jk}}{\partial a} \Big|_a \Big)^2 \lesssim r^{d(1-\varepsilon(1-\beta))}.
\end{align}

\medskip
\emph{Step~4. Sensitivity estimate for all $r \geq r_e$.} The range of radii $r \in [r_e,r_*)$ is covered by applying the result from the previous step to the functional 
\[
\tilde F\psi \coleq \Big( \frac{r}{r_*} \Big)^\frac{d}{2} F\psi.
\]
Since \eqref{eq:a12_needed} holds true, we infer 
\[
\big| \tilde F\psi \big|^2 = \Big( \frac{r}{r_*} \Big)^d |F\psi|^2 \leq \Big( \frac{r}{r_*} \Big)^d \fint_{B_r} |\psi|_a^2 \leq \fint_{B_{r_*}} |\psi|_a^2.
\]
This enables us to employ \eqref{eq:sensitivity_largescale} for $\tilde F$ and $r = r_*$ therein resulting in 
\begin{align}
\label{eq:sensitivity_smallscale}
(r_*)^d \Big( \frac{r}{r_*} \Big)^d \sum_D \Big( \int_D \Big| \frac{\partial F \nabla \phi}{\partial a} \Big|_a \Big)^2 + (r_*)^d \Big( \frac{r}{r_*} \Big)^d \sum_D \Big( \int_D \Big| \frac{\partial F \nabla \sigma_{jk}}{\partial a} \Big|_a \Big)^2 \lesssim (r_*)^{d(1-\varepsilon(1-\beta))}.
\end{align}
Combining \eqref{eq:sensitivity_largescale} and \eqref{eq:sensitivity_smallscale} establishes the announced estimate \eqref{eq:sensitivity}. 
\end{proof}

\subsection{Auxiliary results on the sublinear growth of the extended corrector $(\phi, \sigma)$}
%\begin{lemma}[Higher moments of random variables {\cite[Proposition 1.10]{DG20}}]
%\label{lemmamomentcontrol}
%Assume that the random field $a$ satisfies the multiscale spectral gap (MSG) and the multiscale logarithmic Sobolev inequality (MLSI) from Definition \ref{deffuncineq}. Then, there exists a constant $C = C(\pi,d)$ such that for all $1 \le P < \infty$ and all $\sigma(a)$-measurable random variables $X(a)$, we have
%\begin{align*}
%\big\langle (X(a) - \langle X(a) \rangle)^{2P}\big\rangle \leq (CP)^P \bigg\langle \Big(\int_1^\infty \int_{\R^d} \big(\partial^{\textnormal{fct}}_{x,\ell} \, X(a) \big)^2 \d x \frac{\pi(\ell)}{\ell^d} \,\textnormal{d}\ell \Big)^P\bigg\rangle.
%\end{align*}
%%	Notice that when the weighted is concetrated on a positive $\ell$, that is, when $\pi = \delta_\ell$, we get
%%	\begin{align*}
%%	\mathbb E\big[(X - \mathbb E[X])^{2P}\big]^\frac{1}{P} \le CP \mathbb E\bigg[\Big(\ell^{-d}\int_{\R^d} \big( \partial^{\textnormal{fct}}_{x,\ell} \, X \big)^2 \d x  \Big)^P\bigg]^\frac{1}{P}
%%	\end{align*}
%\end{lemma}

\begin{lemma}
\label{lemmagrowth}
For $0 < L < \rho < \infty$, $1 \le s < d$, $s \leq S < \infty$ such that $\theta \coleq d(\frac 1s - \frac 1S) \in [0,1)$, and $u \in W^{1,s}(B_r)$, we have
	\begin{align}
	\label{eqavu}
	\frac{1}{\rho} \biggl( \fint_{B_\rho} \Big|u - \fint_{B_\rho} u \Big|^{S} \biggr)^{\frac{1}{S}} \le C(d) \biggl( \fint_{B_\rho} \Big|\fint_{B_L(x)} \nabla u \Big|^{s} \d x \biggr)^{\frac{1}{s}} + C(d) \biggl( \frac{L}{\rho} \biggr)^{1-\theta} \biggl( \fint_{B_{2\rho}} |\nabla u|^{s} \biggr)^{\frac{1}{s}}.
	\end{align}
\end{lemma}
\begin{proof}
	By scaling we can assume $\rho=1$, in which case~\eqref{eqavu} reduces for some $L \le 1$ to
	\begin{align}
	\label{avg02}
	\biggl( \int_{B_1} \Big|u - \fint_{B_1} u \Big|^{S} \biggr)^{\frac{1}{S}} \le C(d) \biggl( \int_{B_1} \Big|\fint_{B_L(x)} \nabla u \Big|^{s} \d x \biggr)^{\frac{1}{s}} + C(d) L^{1-\theta} \biggl( \int_{B_{2}} |\nabla u|^{s} \biggr)^{\frac{1}{s}}.
	\end{align}  
	To show this, we apply the triangle inequality to estimate the left-hand side by 
	\begin{align*}
	\label{avg03}
	\biggl( \int_{B_1} \Big|u - \fint_{B_1} u \Big|^{S} \biggr)^{\frac{1}{S}} \le \biggl( \int_{B_1} \Big|(u - u_L) - \fint_{B_1} (u-u_L) \Big|^{S} \biggr)^{\frac{1}{S}} + \biggl( \int_{B_1} \Big|u_L - \fint_{B_1} u_L \Big|^{S} \biggr)^{\frac{1}{S}},
	\end{align*}  
	where $u_L(x) \coleq \fint_{B_L(x)} u$. Combining Jensen's, Sobolev's, and Poincar\'e's inequalities (while using $\frac{1}{S} \ge \frac{1}{s} - \frac 1d$), we get for the second term on the right-hand side
	\begin{align*}
	\biggl( \int_{B_1} \Big|u_L - \fint_{B_1} u_L \Big|^{S} \biggr)^{\frac{1}{S}} &\le C(d) \biggl( \int_{B_1} \Big|u_L - \fint_{B_1} u_L \Big|^{\frac{ds}{d-s}} \biggr)^{\frac{d-s}{ds}} \\ 
	&\le C(d) \biggl( \int_{B_1} |\nabla(u_L)|^{s} \biggr)^{\frac{1}{s}} = C(d) \biggl( \int_{B_1} |(\nabla u)_L)|^{s} \biggr)^{\frac{1}{s}}.
	\end{align*}
	For the first term, we apply H\"older's inequality (with exponents $\frac{s}{s-S(1-\theta)}, \frac{s}{S(1-\theta)} \in (1,\infty)$ after splitting the integrand as $|\cdot|^S = |\cdot|^{\theta S} |\cdot|^{(1-\theta)S}$) and Jensen's inequality, followed in the next step by the above Sobolev inequality and the convolution estimate: 
	\begin{align*}
	&\biggl( \int_{B_1} \Big|(u - u_L) - \fint_{B_1} (u-u_L) \Big|^{S} \biggr)^{\frac{1}{S}} \\
	&\qquad \le C(d) \biggl( \int_{B_1} \Big|(u - u_L) - \fint_{B_1} (u-u_L) \Big|^{\frac{ds}{d-s}} \biggr)^{\theta\frac{d-s}{ds}}\biggl( \int_{B_1} |u - u_L|^s \biggr)^{\frac{1-\theta}{s}}
	\\
	&\qquad \le C(d) \biggl( \int_{B_1} |\nabla(u - u_L)|^s \biggr)^{\frac{\theta}{s}} L^{1-\theta} \biggl( \int_{B_2} |\nabla u|^s \biggr)^{\frac{1-\theta}{s}} 
	\le C(d) L^{1-\theta} \biggl( \int_{B_2} |\nabla u|^s \biggr)^{\frac{1}{s}},
	\end{align*}
	which proves \eqref{avg02}. %For $\theta = 1$, the last estimate follows without the use of H\"older's inequality.
\end{proof}

\begin{corollary}%[Sublinear growth of the correctors]
\label{corgrowth}
For $0 < L < r$, $\theta \coleq \frac d2 \big(\frac 1p + \frac 1q\big) \in (0, 1)$, and 
\[
\bar K \coleq \sup_{R \ge r}\ \bigg( \fint_{B_R} |a|^p \bigg)^{\frac 1p} + \bigg( \fint_{B_R} |a^{-1}|^q \bigg)^{\frac 1q},
\]
the extended corrector $(\phi, \sigma)$ from Definition \ref{defcorrector} satisfies 
\begin{align}
\label{coravg}
&\frac{1}{r} \biggl( \fint_{B_r} \Big|\phi - \fint_{B_r} \phi \Big|^{\frac{2p}{p-1}} \biggr)^{\frac{p-1}{2p}} + \frac{1}{r} \biggl( \fint_{B_r} \Big|\sigma - \fint_{B_r} \sigma \Big|^{\frac{2q}{q-1}} \biggr)^{\frac{q-1}{2q}} \\
&\quad \lesssim  \biggl( \fint_{B_r} \Big|\fint_{B_L(x)} \nabla \phi \Big|^{\frac{2q}{q+1}} \d x \biggr)^{\frac{q+1}{2q}} + \biggl( \fint_{B_r} \Big|\fint_{B_L(x)} \nabla \sigma \Big|^{\frac{2p}{p+1}} \d x \biggr)^{\frac{p+1}{2p}} \nonumber \\ 
&\qquad + \bar K^{\frac{1}{2}} \biggl( \frac{L}{r} \biggr)^{1-\theta}  \biggl( 1 + \frac{1}{r} \biggl( \fint_{B_{8r}} \Big|\phi - \fint_{B_{8r}} \phi \Big|^{\frac{2p}{p-1}} \biggr)^{\frac{p-1}{2p}} + \frac{1}{r} \biggl( \fint_{B_{8r}} \Big|\sigma - \fint_{B_{8r}} \sigma \Big|^{\frac{2q}{q-1}} \biggr)^{\frac{q-1}{2q}} \biggr). \nonumber
\end{align}
\end{corollary}
\begin{proof}
We start with $\sigma$. Using Lemma \ref{lemmagrowth} with $u=\sigma$, $S=\frac{2q}{q-1}$, and $s=\frac{2p}{p+1}$, we get that
	\begin{align}
	\label{avg04}
	&\frac{1}{r} \biggl( \fint_{B_r} \Big|\sigma - \fint_{B_r} \sigma \Big|^{\frac{2q}{q-1}} \biggr)^{\frac{q-1}{2q}} \\ 
	&\qquad \le C(d) \biggl( \fint_{B_r} \Big|\fint_{B_L(x)} \nabla \sigma \Big|^{\frac{2p}{p+1}} \d x \biggr)^{\frac{p+1}{2p}} + C(d) \biggl( \frac{L}{r} \biggr)^{1-\theta} \biggl( \fint_{B_{2r}} |\nabla \sigma|^{\frac{2p}{p+1}} \biggr)^{\frac{p+1}{2p}}, \nonumber
	\end{align}    
	with $\theta = \frac d2 (\frac 1p + \frac 1q) < 1$. To estimate the second term on the right-hand side, we assume w.l.o.g.\ that $\fint_{B_{4r}} \sigma = 0$ and consider $\eta \sigma$ with a smooth cut-off function $\eta$ for $B_{2r}$ in $B_{4r}$, which then by $-\Delta \sigma_{ijk} = \partial_j q_{ik} - \partial_k q_{ij} \eqcol \nabla \cdot \tilde q$ with $\tilde q = q_{ik}e_j - q_{ij}e_k$ satisfies
	\begin{align*}
	\Delta (\eta \sigma_{ijk}) = \nabla \cdot ( 2\sigma_{ijk} \nabla \eta - \eta \tilde q) + \nabla \eta \cdot \tilde q - \sigma_{ijk} \Delta \eta. 
	\end{align*}
	By a Calder\'on--Zygmund estimate for the Laplacian (see, e.g., \cite[Subsections 7.1.2--7.1.3]{GM12}) combined with Sobolev's and Jensen's inequalities, this implies 
	\begin{align*}
	\biggl( \int_{B_{2r}} |\nabla \sigma_{ijk}|^{\frac{2p}{p+1}} \biggr)^{\frac{p+1}{2p}} &\le
	\biggl( \int_{\R^d} |\nabla(\eta \sigma_{ijk})|^{\frac{2p}{p+1}} \biggr)^{\frac{p+1}{2p}} 
	\lesssim 
	\frac 1r \biggl( \int_{B_{4r}} |\sigma_{ijk}|^{\frac{2p}{p+1}} \biggr)^{\frac{p+1}{2p}} + 
	\biggl( \int_{B_{4r}} |\tilde q|^{\frac{2p}{p+1}} \biggr)^{\frac{p+1}{2p}}.
	\end{align*}
	Using the definition of $\tilde q$ via $q_i=a(\nabla \phi_i + e_i)$ we see that by H\"older's inequality
	\begin{align*}
	\biggl( \fint_{B_{4r}} |\tilde q|^{\frac{2p}{p+1}} \biggr)^{\frac{p+1}{2p}} \le
	\biggl( \fint_{B_{4r}} |a|^p \biggr)^{\frac{1}{2p}} 
	\biggl( \fint_{B_{4r}} |\nabla \phi_i + e_i|_a^2 \biggr)^{\frac{1}{2}},
	\end{align*}
	which combined with the previous inequality (after taking averages on both sides) yields
	\begin{equation*}
	\biggl( \fint_{B_{2r}} |\nabla \sigma_{ijk}|^{\frac{2p}{p+1}} \biggr)^{\frac{p+1}{2p}} 
	\lesssim 
	\frac 1r \biggl( \fint_{B_{4r}} |\sigma_{ijk}|^{\frac{2p}{p+1}} \biggr)^{\frac{p+1}{2p}} + 
	\bar K^{\frac{1}{2}} \biggl( \fint_{B_{4r}} |\nabla \phi_i + e_i|_a^2 \biggr)^{\frac{1}{2}}.
	\end{equation*}
	Note that the same arguments apply if we replace $\sigma$ by $\sigma - \fint_{B_{4r}} \sigma$. We then plug the previous estimate into \eqref{avg04} and use Jensen's inequality to obtain 
	\begin{align}
	\label{avg04+}
	&\frac{1}{r} \biggl( \fint_{B_r} \Big|\sigma - \fint_{B_r} \sigma \Big|^{\frac{2q}{q-1}} \biggr)^{\frac{q-1}{2q}} \le C(d) \biggl( \fint_{B_r} \Big|\fint_{B_L(x)} \nabla \sigma \Big|^{\frac{2p}{p+1}} \d x \biggr)^{\frac{p+1}{2p}} \\
	&\quad + C(d) \biggl( \frac{L}{r} \biggr)^{1-\theta} \frac1r \biggl( \fint_{B_{4r}} \Big|\sigma - \fint_{B_{4r}} \sigma \Big|^{\frac{2q}{q-1}} \biggr)^{\frac{q-1}{2q}} +
	C(d) \biggl( \frac{L}{r} \biggr)^{1-\theta} \bar K^{\frac 12} \biggl( \fint_{B_{4r}} |\nabla \phi + e|_a^2 \biggr)^{\frac{1}{2}}. \nonumber
	\end{align}
	To control the last term on the right-hand side, we observe that testing $-\nabla \cdot a \big( \nabla (\phi - \fint_{B_{8r}} \phi) + e \big) = 0$ with $\eta^2 (\phi - \fint_{B_{8r}} \phi)$, where $\eta$ is a cut-off function for $B_{4r}$ in $B_{8r}$, entails 
	\begin{align}
	\label{avg05}
	\fint_{B_{8r}} \Big|\nabla \Big(\eta \big(\phi - \fint_{B_{8r}} \phi \big) \Big) \Big|_a^2 &\lesssim \fint_{B_{8r}} |e|_a^2 + \fint_{B_{8r}} \Big| \phi - \fint_{B_{8r}} \phi \Big|^2 |\nabla \eta|_a^2 \\
	&\lesssim \bar K \biggl( 1 + \frac{1}{r^2} \biggl( \fint_{B_{8r}} \Big|\phi - \fint_{B_{8r}} \phi \Big|^{\frac{2p}{p-1}} \biggr)^\frac{p-1}{p} \biggr), \nonumber
	\end{align}
	where we used H\"older's inequality together with $|\nabla \eta| \lesssim \frac 1r$. 
	
	Using Lemma \ref{lemmagrowth} with $u=\phi$, $S=\frac{2p}{p-1}$ and $s=\frac{2q}{q+1}$, we get that  
	\begin{align}
	\label{avg06}
	&\frac{1}{r} \biggl( \fint_{B_r} \Big|\phi - \fint_{B_r} \phi \Big|^{\frac{2p}{p-1}} \biggr)^{\frac{p-1}{2p}} \\ 
	&\qquad \le C(d) \biggl( \fint_{B_r} \Big|\fint_{B_L(x)} \nabla \phi \Big|^{\frac{2q}{q+1}} \d x \biggr)^{\frac{q+1}{2q}} + C(d) \biggl( \frac{L}{r} \biggr)^{1-\theta} \biggl( \fint_{B_{2r}} |\nabla \phi|^{\frac{2q}{q+1}} \biggr)^{\frac{q+1}{2q}}. \nonumber
	\end{align}    
	By means of H\"older's inequality, we see that 
	\[
	\bigg( \fint_{B_{2r}} \Big|\nabla \Big(\phi - \fint_{B_{8r}} \phi \Big) \Big|^{\frac{2q}{q+1}} \bigg)^{\frac{q+1}{2q}} \lesssim \bar K^{\frac 12} \bigg( \fint_{B_{8r}} \Big|\nabla \Big(\eta \big(\phi - \fint_{B_{8r}} \phi \big) \Big) \Big|_a^2 \bigg)^{\frac 12},
	\]
	which we use in~\eqref{avg06} to control the last integral. Finally, such modified~\eqref{avg06} together with \eqref{avg05} and \eqref{avg04+} yield \eqref{coravg}, which completes the proof.
\end{proof}

\subsection{Proof of Theorem \ref{theoremmoments}: Stretched exponential moments for $r_\ast$}

\begin{proof}[Proof of Theorem \ref{theoremmoments}]
\emph{Step~1. Control of the minimal radius $r_*$.} Let $p, q \in (1,\infty)$ be the integrability exponents of $a$ and $a^{-1}$ from Definition \ref{defadmissible}. %and define $s \coleq \frac{2q}{q+1} \in (1,2)$, $S \coleq \frac{2p}{p-1} \in (2, \infty)$, and $\theta \coleq d (\frac1s - \frac1S) = \frac{d}{2} \big( \frac{1}{p} + \frac{1}{q} \big) < 1$. 
We shall follow the lines of \cite{GNO20v2} to derive an estimate for the minimal radius $r_*$. 
To this end, we first assume that $M_0 r_e < r \leq r_*$ holds with some positive radius $r$ and the positive constant $M_0$ from Definition \ref{defminradius} which is specified in \eqref{eqdefm0} below. We introduce 
\[
X(r) \coleq \max \bigg\{ \frac 1r \Big( \fint_{B_r} \Big| \phi - \fint_{B_r} \phi \Big|^\frac{2p}{p-1} \Big)^\frac{p-1}{2p}, \ \frac 1r \Big( \fint_{B_r} \Big| \sigma - \fint_{B_r} \sigma \Big|^\frac{2q}{q-1} \Big)^\frac{q-1}{2q} \bigg\}. %\ \ Y(r, x) \coleq \fint_{B_r(x)} \nabla (\phi, \sigma).
\]
%and 
%\[
%r_*^\mathrm{pure} \coleq \inf \bigg\{ r \geq 1 \: \Big| \: \forall \, \rho > r : \max \bigg\{ \frac1{\rho} \Big( \fint_{B_\rho} \Big| \phi - \fint_{B_\rho} \phi \Big|^\frac{2p}{p-1} \Big)^\frac{p-1}{2p}, \frac1{\rho} \Big( \fint_{B_\rho} \Big| \sigma - \fint_{B_\rho} \sigma \Big|^\frac{2q}{q-1} \Big)^\frac{q-1}{2q} \bigg\} \leq \frac{1}{C_0} \bigg\}.
%\]
%The only difference between $r_*$ and $r_*^\mathrm{pure}$ is the lower bound in the respective definition, namely $r_* \geq r_e$ and $r_*^\mathrm{pure} \geq 1$. 
With another positive radius $r'$ subject to $r' < r$, we may employ Corollary \ref{corgrowth} to obtain 
\begin{align*}
%\label{eqXYesimate}
X(r) \lesssim \bigg( \Big( \frac{r'}{r} \Big)^{1-\theta} \big( 1 + X(8r) \big) + \Big( \fint_{B_r} \Big|\fint_{B_{r'}(x)} \nabla \phi \Big|^\frac{2p}{p+1} \Big)^\frac{p+1}{2p} + \Big( \fint_{B_r} \Big|\fint_{B_{r'}(x)} \nabla \sigma \Big|^\frac{2q}{q+1} \Big)^\frac{q+1}{2q} \bigg)
\end{align*}
with $\theta = \frac{d}{2} \big( \frac{1}{p} + \frac{1}{q} \big) < 1$. From the definition of $r_*$ and the constant $C_0$ in Definition \ref{defminradius}, we deduce that there exists some $\rho \geq r$ such that $X(\rho) \geq \frac{1}{C_0}$ while $X(\rho') \leq \frac{2}{C_0}$ for all $\rho' \geq \rho$. Note that we crucially employed the assumption $M_0 r_e < r_*$, which ensures that such a $\rho \geq r$ actually exists. We infer 
\[
X(\rho') \geq \Big( \frac{\rho}{\rho'} \Big)^{1 + d \max\big\{ \tfrac{p-1}{2p}, \tfrac{q-1}{2q} \big\}} X(\rho) \geq \Big(\frac{1}{2} \Big)^{1 + d \max\big\{ \tfrac{p-1}{2p}, \tfrac{q-1}{2q} \big\}} \frac{1}{C_0}, \qquad X(8\rho') \leq \frac{2}{C_0}
\]
for all $\rho' \in (\rho, 2\rho)$. The choice $r \coleq \rho'$ and $r' \coleq \rho'' \in (0, \rho')$ entails 
\[
\frac{1}{C} \leq \Big( \frac{\rho''}{\rho'} \Big)^{1-\theta} + \Big( \fint_{B_{\rho'}} \Big|\fint_{B_{\rho''}(x)} \nabla \phi \Big|^\frac{2p}{p+1} \Big)^\frac{p+1}{2p} + \Big( \fint_{B_{\rho'}} \Big|\fint_{B_{\rho''}(x)} \nabla \sigma \Big|^\frac{2q}{q+1} \Big)^\frac{q+1}{2q} 
\]
%where $C > 1$ is such that $X(\rho') \geq \frac1C$. 
with a positive constant $C(d, p, q, K, C_0)$ which is in particular independent of $M_0$. 
One can now absorb the first term on the right-hand side by setting $\rho'' \coleq \frac{\rho'}{M_0}$ with 
\begin{align}
\label{eqdefm0}
M_0 \coleq ( 2 C )^\frac{1}{1-\theta}
\end{align}
leading to 
\[
\frac1{2C} \leq \Big( \fint_{B_{\rho'}} \Big|\fint_{B_{\frac{\rho'}{M_0}}(x)} \nabla \phi \Big|^\frac{2p}{p+1} \Big)^\frac{p+1}{2p} + \Big( \fint_{B_{\rho'}} \Big|\fint_{B_{\frac{\rho'}{M_0}}(x)} \nabla \sigma \Big|^\frac{2q}{q+1} \Big)^\frac{q+1}{2q}.
\]
Taking the power $2P$, $P \in \mathbb N$, applying Jensen's inequality, and integrating over $(\rho, 2\rho)$, we find 
\[
\frac1{4(4C)^{2P}} \leq \int_{\rho}^{2\rho} \fint_{B_{\rho'}} \Big|\fint_{B_{\frac{\rho'}{M_0}}(x)} \nabla (\phi, \sigma) \Big|^{2P} \frac{1}{\rho'}.
\]
Since the previous calculation holds for any configuration satisfying $M_0 r_e < r \leq r_*$, we arrive at 
\begin{align}
\label{eqminradestimate}
\langle I(M_0 r_e < r \leq r_*) \rangle \leq C^P \int_\frac{r}{M_0}^\infty \Big\langle \Big| \fint_{B_\rho} \nabla (\phi, \sigma) \Big|^{2P} \Big\rangle \frac{1}{\rho}
\end{align}
extending the range of integration, renaming variables, and using the stationarity of $\fint_{B_\rho(x)}\!\nabla (\phi, \sigma)$.

\emph{Step~2. Control of the corrector gradient $\nabla (\phi,\sigma)$.} 
We keep the assumptions and the notation from the previous step and consider some $\rho > \tfrac{r}{M_0}$. As a result of the vanishing expectation of $\nabla (\phi,\sigma)$, the spectral gap estimate \eqref{eq:pspectralgap} ensures the bound 
\[
\Big\langle \Big| \fint_{B_\rho} \nabla (\phi,\sigma) \Big|^{2P} \Big\rangle^\frac1P \leq \frac{CP^2}{\kappa} \bigg\langle \Big( \sum_D \Big( \int_D \Big| \frac{\partial}{\partial a} \fint_{B_\rho} \nabla (\phi,\sigma) \Big|_a \Big)^2 \Big)^P \bigg\rangle^\frac1P.
\]
Recalling the notation $F\nabla (\phi,\sigma) = \int_{\mathbb R^d} g \nabla (\phi,\sigma) = \fint_{B_\rho} \nabla (\phi,\sigma)$ with $g \coleq |B_\rho|^{-1} \mathbb 1_{B_\rho}$, which in particular fulfills the condition $\big( \fint_{B_\rho} |g|^\frac{2p}{p-1} \big)^\frac{p-1}{2p} %\lesssim \rho^{-d} \big( \fint_{B_\rho} |a^{-1}| \big)^\frac12 
\lesssim \rho^{-d}$, Proposition \ref{propsensitivity} guarantees 
\[
\Big\langle \Big| \fint_{B_\rho} \nabla (\phi,\sigma) \Big|^{2P} \Big\rangle^\frac1P \lesssim P^2 \Big\langle \Big( \frac{(\rho + r_*)^{1-\varepsilon(1-\beta)}}{\rho} \Big)^{dP} \Big\rangle^\frac1P. 
\]
Here and in the remainder of this proof, all generic constants may depend on $d$, $p$, $q$, $\beta$, $\varepsilon$, $\kappa$, $K$, and $C_0$. Together with the estimate on the minimal radius in \eqref{eqminradestimate}, we derive 
\begin{align*}
\big\langle I(M_0 r_e < r \leq r_*) \big\rangle^\frac1P &\lesssim \bigg( \int_\frac{r}{M_0}^\infty \Big\langle \Big| \fint_{B_\rho} \nabla (\phi,\sigma) \Big|^{2P} \Big\rangle \frac{1}{\rho} \bigg)^\frac1P \\ 
&\lesssim P^2 \bigg( \int_\frac{r}{M_0}^\infty \Big\langle \Big( \frac{(\rho + r_*)^{1-\varepsilon(1-\beta)}}{\rho} \Big)^{dP} \Big\rangle \frac{1}{\rho} \bigg)^\frac1P.
\end{align*}
An evaluation of the integral over $\rho$ after using the triangle inequality gives rise to 
\begin{align*}
\big\langle I(M_0 r_e < r \leq r_*) \big\rangle^\frac1P &\lesssim P^2 \bigg( \Big( \frac{1}{\varepsilon d P (1-\beta)} \Big( \frac{r}{M_0} \Big)^{-\varepsilon d P (1-\beta)} \Big)^\frac1P \\ 
&\qquad + \Big( \frac{1}{dP} \Big( \frac{r}{M_0} \Big)^{-dP} \Big)^\frac1P \big\langle r_*^{dP(1-\varepsilon(1-\beta))} \big\rangle^\frac1P \bigg).
\end{align*}
Keeping track only of the dependence of the constants on $P$, this expression simplifies to 
\begin{align*}
\big\langle I(M_0 r_e < r \leq r_*) \big\rangle^\frac1P \lesssim P^2 \Big( \frac{1}{r^{\varepsilon d (1-\beta)}} + \frac{1}{r^d} \big\langle r_*^{dP(1-\varepsilon(1-\beta))} \big\rangle^\frac1P \Big).
\end{align*} 
To derive a bound on $\langle I(r_* \geq r) \rangle$, we first note that 
\[
\langle I(r_* \geq r) \rangle^\frac1P \leq \langle I(M_0 r_e \geq r) \rangle^\frac1P + \langle I(M_0 r_e < r \leq r_*) \rangle^\frac1P
\]
by elementary arguments, where $\langle I(M_0 r_e \geq r) \rangle^\frac1P \lesssim \exp \big(\!-\frac{1}{CP} \big( \frac{r}{M_0} \big)^{\varepsilon \frac{d}{2} (1-\beta)} \big)$ due to \eqref{eq:strexpmom_re} with $\alpha = \frac{\varepsilon}{1 - \varepsilon}$. As a consequence of the scalar inequality $e^{-x} \leq x^{-2}$ for any $x > 0$, we obtain 
\[
\exp \Big(\!-\frac{1}{CP} \Big( \frac{r}{M_0} \Big)^{\varepsilon \frac{d}{2} (1-\beta)} \Big) \leq \frac{C^2 P^2 M_0^{\varepsilon d (1-\beta)}}{r^{\varepsilon d (1-\beta)}}. 
\]
Hence, we established the estimate 
\begin{align}
\label{eq:minradbuckling_r}
\langle I(r_* \geq r) \rangle^\frac1P \lesssim P^2 \Big( \frac{1}{r^{\varepsilon d (1-\beta)}} + \frac{1}{r^d} \Big\langle r_*^{dP(1-\varepsilon (1-\beta))} \Big\rangle^\frac1P \Big).
\end{align}

\emph{Step~3. Buckling and exponential moments of $r_*$.} We introduce $s \coleq r^{\frac\varepsilon2 d (1-\beta)}$, $s_* \coleq r_*^{\frac\varepsilon2 d (1-\beta)}$, $Q \coleq \frac{1}{\varepsilon(1-\beta)} > 1$, and we replace $P \geq 1$ by $\tfrac P2$ (now with $P \geq 2$) in \eqref{eq:minradbuckling_r}. 
The transformed estimate then reads 
\begin{align}
\label{eq:buckling_start}
\langle I(s_* \geq s) \rangle^\frac1P \leq \wt C P \Big( \frac{1}{s} + \frac{1}{s^Q} \big\langle (s_*^{Q-1})^P \big\rangle^\frac1P \Big)
\end{align}
with some constant $\wt C \geq 1$. The basic idea of buckling in this context is to assume that there exists a constant $\Lambda \geq 1$ such that 
\begin{align}
\label{eq:bucklingassump}
\big\langle I(s_* \geq s) \big\rangle \leq \exp \Big(\!-\frac{s}{\Lambda} \Big)
\end{align}
holds true for all $s \geq \Lambda$. This assumption is in general only satisfied for %an approximating sequence $\varphi_\delta(s) \rightarrow \big\langle I(s_* \geq s) \big\rangle$, $\delta \rightarrow 0$. 
$\min \{ s_*, \delta^{-1} \}$ with $\delta > 0$ instead of $s_*$. Showing that \eqref{eq:bucklingassump} (with $\min \{ s_*, \delta^{-1} \}$) entails $\Lambda \leq C$ for some constant $C \geq 1$, which is in particular independent of $\delta$, allows to transfer the uniform exponential decay to $\big\langle I(s_* \geq s) \big\rangle$.

Imposing \eqref{eq:bucklingassump}, we calculate for any $P \geq 2$, 
\begin{align*}
\langle s_*^P \rangle &\leq \int_0^\infty \big\langle I(s_* \geq s) \big\rangle \frac{\d}{\d s} s^P \, \d s \leq \int_0^\infty \exp \Big( 1-\frac{s}{\Lambda} \Big) \frac{\d}{\d s} s^P \, \d s \\
&= e \int_0^\infty \frac1\Lambda \exp \Big(\!-\frac{s}{\Lambda} \Big) s^P \, \d s = e \int_0^\infty e^{-t} (\Lambda t)^P \, \d t = e \Lambda^P P!.
\end{align*}
We thus obtain $\langle s_*^P \rangle^\frac 1P \leq e P \Lambda$ for all $P \geq 2$ and infer from \eqref{eq:buckling_start} the estimate 
\[
\langle I(s_* \geq s) \rangle^\frac1P \leq \wt C P \Big( \frac 1s + \frac{1}{s^Q} \big(e P (Q-1) \Lambda \big)^{Q-1} \Big) \leq C_0 \frac P s \Big( 1 + \Big( \frac{P \Lambda}{s} \Big)^{Q-1} \Big)
\]
with another constant $C_0 \geq 1$. By defining $\Lambda' \coleq 2C_0\Lambda^\frac{Q-1}{Q} \geq 1$, we have 
\[
\langle I(s_* \geq s) \rangle^\frac1P \leq C_0 \frac P s \bigg( \frac{\Lambda'}{2C_0} + \bigg( \frac{\Lambda}{2C_0 \Lambda^\frac{Q-1}{Q}} \bigg)^{Q-1} \bigg) \leq \frac P s \Lambda'
\]
provided $s \geq P \Lambda'$. The best upper bound for $\langle I(s_* \geq s) \rangle \leq \big( \frac P s \Lambda' \big)^P$ can be found by optimizing the right-hand side in $P$ leading to $P = \tfrac{s}{e\Lambda'}$. In order to guarantee $P \geq 2$ (due to the variable transformation at the beginning of this step), we have to demand $s \geq 2 e \Lambda'$. Note that $s \geq P\Lambda'$ is trivially satisfied. Consequently, 
\[
\langle I(s_* \geq s) \rangle \leq e^{-P} \leq \exp \Big(\!-\frac{s}{2 e \Lambda'} \Big)
\]
for all $s \geq 2 e \Lambda'$. This shows that \eqref{eq:bucklingassump} also holds true with $2 e \Lambda'$ instead of $\Lambda$. Taking the best constant $\Lambda$ in \eqref{eq:bucklingassump} ensures $\Lambda \leq 2 e \Lambda' = 4 C_0 e \Lambda^\frac{Q-1}{Q}$ and, finally, $\Lambda \leq (4 C_0 e)^Q$. 

Moments of $s_*$ of order $k \in \mathbb N$ are now derived as above via 
\begin{align*}
\langle s_*^k \rangle &\leq \int_0^\infty \langle I(s_* \geq s) \rangle \frac{\d}{\d s} s^k \, \d s \leq \int_0^\infty \exp\Big(\!-\frac{s}{\Lambda}\Big) \frac{\d}{\d s} s^k \, \d s \\ 
&= \int_0^\infty \frac1\Lambda \exp \Big(\!-\frac{s}{\Lambda} \Big) s^k \, \d s = \int_0^\infty e^{-t} (\Lambda t)^k \, \d t = \Lambda^k k!.
\end{align*}
As an immediate consequence, we obtain stretched exponential moments of $r_*$ in the sense that 
\[
\Big\langle \exp \Big( \frac1C r_*^{\frac{\varepsilon}{2} d (1-\beta)} \Big) \Big\rangle = 1 + \sum_{k=1}^\infty \frac1{k!} \frac{\langle s_*^k \rangle}{C^k} \leq 1 + \sum_{k=1}^\infty \frac{\Lambda^k}{C^k} < 2
\]
for a sufficiently large constant $C > 0$.
\end{proof}

\section{Two applications to stochastic homogenization}
\subsection{Proof of Theorem \ref{theoremcorrectors}: Decay and growth properties of the extended corrector}

\begin{proof}[Proof of Theorem \ref{theoremcorrectors}]
We define the random variable 
\[
F(x) \coleq \fint_{B_r} \nabla (\phi, \sigma)(x+y) \cdot m(y) \, dy
\]
which is stationary and satisfying $\langle F \rangle = 0$ thanks to the according properties of $\nabla (\phi, \sigma)$. From the $P$-spectral gap inequality \eqref{eq:pspectralgap} and the sensitivity estimate \eqref{eq:sensitivity}, we thus derive 
\[
\big\langle |F|^{2P} \big\rangle^\frac{1}{2P} \lesssim P \bigg\langle \Big( \sum_D \Big( \int_D \Big| \frac{\partial F}{\partial a} \Big|_a \Big)^2 \Big)^P \bigg\rangle^\frac{1}{2P} \lesssim P \Big\langle \Big( \frac{(r+r_*)^{1-\varepsilon(1-\beta)}}{r} \Big)^{Pd} \Big\rangle^\frac{1}{2P} 
\]
for any $P \in \mathbb N$, $P \geq 2$. Taking $r \geq 1$ and $r_* \geq 1$ into account, this simplifies to 
\begin{align}
\label{eq:2pspectralgap}
\big\langle |F|^{2P} \big\rangle^\frac{1}{2P} \lesssim P r^{-\frac\varepsilon2 d (1-\beta)} \Big\langle r_*^{(1-\varepsilon(1-\beta)) Pd} \Big\rangle^\frac{1}{2P}.
\end{align}
The first part of the theorem is a consequence of establishing stretched exponential moments %of order $\gamma$ 
for 
\[
\mathcal C(x) \coleq r^{\frac \varepsilon 2 d (1 - \beta)} \big| F(x) \big|. 
\]
To this end, we choose $Q \geq 1$ and let $P \in \mathbb N$ be the integer such that $Q \leq P < Q + 1$. Estimate \eqref{eq:2pspectralgap} then entails 
\[
\big\langle |\mathcal C|^{2Q} \big\rangle^\frac{1}{2Q} \lesssim Q \Big\langle r_*^{(1-\varepsilon(1-\beta))(Q+1)d} \Big\rangle^\frac{1}{2(Q+1)}.
\]
Replacing $Q$ by $\tfrac{\gamma Q}{2}$ for some $\gamma > 0$ and demanding $Q \geq \tfrac 4 \gamma$ to guarantee $\tfrac{\gamma Q}{2} \geq 2$, we obtain 
\[
\big\langle |\mathcal C|^{\gamma Q} \big\rangle^\frac{1}{Q} \lesssim Q^\gamma \Big\langle r_*^{(1-\varepsilon(1-\beta))(\frac{\gamma Q}{2}+1)d} \Big\rangle^\frac{\gamma}{2(\frac{\gamma Q}{2} + 1)} \leq Q^\gamma \Big\langle r_*^{(1-\varepsilon(1-\beta))\gamma Q d} \Big\rangle^\frac{1}{2Q}.
\]
Together with Theorem \ref{theoremmoments} and Lemma \ref{lemma:expmom_polymom}, we derive the bound 
\[
\Big\langle r_*^{(1-\varepsilon(1-\beta))\gamma Q d} \Big\rangle^\frac{1}{2Q} \lesssim Q^\frac{\gamma (1 - \varepsilon (1 - \beta))}{\varepsilon (1 - \beta)},
\]
which results in 
\[
\big\langle |\mathcal C|^{\gamma Q} \big\rangle^\frac{1}{Q} \lesssim Q^{\gamma \big( 1 + \frac{(1 - \varepsilon (1 - \beta))}{\varepsilon (1 - \beta)} \big)} = Q^\frac{\gamma}{\varepsilon (1 - \beta)} = Q
\]
for $\gamma \coleq \varepsilon (1 - \beta)$. Applying again Lemma \ref{lemma:expmom_polymom} concludes the argument. 

Concerning the growth of the correctors $(\phi, \sigma)$, we start by employing Sobolev's inequality related to the embedding $W^{1,\frac{2q}{q+1}} \big( B(x) \big) \hookrightarrow L^\frac{2p}{p-1} \big( B(x) \big)$ on the unit ball $B(x) \coleq B_1(x)$, which holds by the condition $\frac{1}{p} + \frac{1}{q} \leq \frac{2}{d}$. We infer  
\[
\Big( \fint_{B(x)} | \phi |^\frac{2p}{p-1} \Big)^\frac{p-1}{2p} \lesssim \Big( \fint_{B(x)} | \phi |^\frac{2q}{q+1} \Big)^\frac{q+1}{2q} + \Big( \fint_{B(x)} | \nabla \phi |^\frac{2q}{q+1} \Big)^\frac{q+1}{2q}.
\]
Poincar\'e's inequality and an elementary estimate involving the ellipticity radius $r_e(x)$ give rise to 
\[
\Big( \fint_{B(x)} | \phi |^\frac{2p}{p-1} \Big)^\frac{p-1}{2p} \lesssim \Big| \fint_{B(x)} \phi \Big| + r_e(x)^{d \frac{q+1}{2q}} \Big( \fint_{B_{r_e(x)}(x)} | \nabla \phi |^\frac{2q}{q+1} \Big)^\frac{q+1}{2q}.
\]
H\"older's inequality followed by a hole-filling argument then yields 
\[
\Big( \fint_{B(x)} | \phi |^\frac{2p}{p-1} \Big)^\frac{p-1}{2p} \lesssim \Big| \fint_{B(x)} \phi \Big| + r_e(x)^{d \frac{q+1}{2q}} \Big( \frac{r_*(x)}{r_e(x)} \Big)^{\frac d2 (1 - \varepsilon)} \Big( \fint_{B_{r_*(x)}(x)} |\nabla \phi|_a^2 \Big)^\frac12. 
\]
Owing to a Caccioppoli estimate of the last term above and the definition of the minimal radius $r_*(x)$ from Definition \ref{defminradius}, we deduce 
\begin{align*}
\Big( \fint_{B(x)} | \phi |^\frac{2p}{p-1} \Big)^\frac{p-1}{2p} &\lesssim \Big| \fint_{B(x)} \phi \Big| + r_e(x)^{d \frac{q+1}{2q}} \Big( \frac{r_*(x)}{r_e(x)} \Big)^{\frac d2 (1 - \varepsilon)} \Big( \fint_{B_{r_*(x)}(x)} \Big|\phi - \fint_{B_{r_*(x)}(x)} \phi \Big|^\frac{2p}{p-1} \Big)^\frac{p-1}{2p} \\ 
&\lesssim \Big| \fint_{B(x)} \phi \Big| + r_*(x)^{\frac d2 (1 - \varepsilon)} r_e(x)^{\frac d2 ( \frac{1}{q} + \varepsilon )}.
\end{align*}
As in \cite{GNO20v3}, we proceed by calculating 
\[
\fint_{B_t(x)} \big( \phi(y) - \phi(x) \big) \, dy = \int_0^t \fint_{B_r(x)} \nabla \phi (y) \cdot \frac{y - x}{r} \, dy \, dr
\]
for any $t \geq 1$. This identity easily follows from the integral mean value theorem. Hence, 
\[
\partial_t \fint_{B_t(x)} \phi = \fint_{B_t(x)} \nabla \phi (y) \cdot \frac{y - x}{t} \, dy.
\]
Moreover, we have $\fint_{B_t(x)} \big| \frac{y-x}{t} \big|^2 \, dy \simeq 1$, which allows us to employ the first part of this theorem ensuring that 
\[
\Big| \partial_t \fint_{B_t(x)} \phi \Big| \leq \mathcal C(x) t^{-\frac \varepsilon 2 d (1 - \beta)}
\]
and, consequently, 
\[
\Big| \fint_{B_R(x)} \phi - \fint_{B(x)} \phi \Big| \leq \int_1^R \Big| \partial_t \fint_{B_t(x)} \phi \Big| \leq \mathcal C(x) \int_1^R t^{-\frac \varepsilon 2 d (1 - \beta)} \lesssim \mathcal C(x) \pi(R)
\]
recalling the definition of $\pi(r)$ from \eqref{eq:defpi}. In a similar fashion, we find that 
\begin{align*}
\fint_{B_R(x)} \phi - \fint_{B_R} \phi &= \fint_{B_R} \big( \phi(x+y) - \phi(y) \big) \, dy \\ 
&= \fint_{B_R} \int_0^1 \nabla \phi (y + tx) \cdot x \, dt \, dy = |x| \int_0^1 \fint_{B_R} \nabla \phi (y + tx) \cdot \frac{x}{|x|} \, dt \, dy.
\end{align*}
Since $\fint_{B_R} \big| \frac{x}{|x|} \big|^2 = 1$, we conclude as above that 
\[
\Big| \fint_{B_R(x)} \phi - \fint_{B_R} \phi \Big| \leq R^{-\frac \varepsilon 2 d (1 - \beta)} |x| \int_0^1 \mathcal C(tx) \, dt.
\]
We may now apply the previous estimates to the right-hand side of the following inequality, 
\[
\Big| \fint_{B(x)} \phi \Big| \leq \Big| \fint_{B} \phi \Big| + \Big| \fint_{B} \phi - \fint_{B_R} \phi \Big| + \Big| \fint_{B_R} \phi - \fint_{B_R(x)} \phi \Big| + \Big| \fint_{B_R(x)} \phi - \fint_{B(x)} \phi \Big|, 
\]
leading to 
\[
\Big| \fint_{B(x)} \phi \Big| \leq \Big| \fint_{B} \phi \Big| + \Big( \mathcal C(0) + \mathcal C(x) + \int_0^1 \mathcal C(tx) \, dt \Big) \pi(|x|)
\]
when choosing $R \coleq |x|$. Together, we have 
\[
\Big( \fint_{B(x)} | \phi |^\frac{2p}{p-1} \Big)^\frac{p-1}{2p} \lesssim r_*(x)^\frac{d}{2} + r_e(x)^{\frac{d}{2} ( \frac{1}{\varepsilon q} + 1)} + \Big| \fint_{B} \phi \Big| + \Big( \mathcal C(0) + \mathcal C(x) + \int_0^1 \mathcal C(tx) \, dt \Big) \pi(|x|).
\]
Observe that $r_*(x)^\frac{d}{2}$ is stochastically integrable with stretched exponential moment $\varepsilon (1 - \beta)$ as $\big\langle \exp \big( \frac1C r_*^{\frac{\varepsilon}{2} d (1-\beta)} \big) \big\rangle < 2$ due to \eqref{eq:strexpmom_rast} and that $r_e(x)^{\frac{d}{2} ( \frac{1}{\varepsilon q} + 1)}$ is also integrable with the same moment $\varepsilon (1 - \beta)$ since $\big( \frac{1}{\varepsilon q} + 1 \big) \varepsilon \leq \frac{\alpha}{\alpha + 1}$ and $\big\langle \exp \big( \frac1C r_e^{\frac{\alpha}{\alpha + 1} \frac d2 (1-\beta)} \big) \big\rangle < 2$ due to \eqref{eq:strexpmom_re}. This closes the proof for the second part of the theorem as the same arguments also apply to $\sigma$ by basically exchanging $p$ and $q$.
\end{proof}

\subsection{Proof of Corollary \ref{cor:twoscale}: A quantitative two-scale expansion}
\begin{proof}[Proof of Corollary \ref{cor:twoscale}]
We follow the strategy presented in \cites{GNO20v3,GNO21} and first notice that by a scaling argument %the integrals on the left- and right-hand side of \eqref{eq:twoscaleerror} scale with $\delta^{-1}$ and $\delta^{-3}$, respectively. 
it is sufficient to prove the claim for $\delta = 1$. 
%\[
%\Big( \int \big| \nabla u_\delta - \big( u_{\mathrm{hom},\delta} + \delta \phi_i \big( \tfrac{\cdot}{\delta} \big) \partial_i u_{\mathrm{hom},\delta} \big) \big|_a^2 \Big)^\frac12 \leq \mathcal C_\delta (g) \Big( \int \pi(|x|)^2 |D^2 u_\mathrm{hom}|^2 \Big)^\frac12 \pi(\delta^{-1}) \delta
%\]
Next, we derive the following equation for $z$ where $(\cdot)_1$ denotes averaging over a ball of radius $1$: 
\begin{align}
\label{eq:twoscaleformula}
-\nabla \cdot a \nabla z = \nabla \cdot (g - g_1 + (a\phi_i - \sigma_i) \nabla \partial_i u_{\mathrm{hom},1}). 
\end{align}
The calculation is carried out in \cite{GNO21} but we recall the main steps of the proof for completeness. From the representation $z \coleq u - \big( u_{\mathrm{hom},1} + \phi_i \partial_i u_{\mathrm{hom},1} \big)$, we obtain 
\[
a \nabla z = a \nabla u - \partial_i u_{\mathrm{hom},1} a (\nabla \phi_i + e_i) - a \phi_i \nabla \partial_i u_{\mathrm{hom},1}
\]
and 
\[
-\nabla \cdot a \nabla z = \nabla \cdot (g - g_1) + \nabla \cdot (a \phi_i \nabla \partial_i u_{\mathrm{hom},1}) + \nabla \partial_i u_{\mathrm{hom},1} \cdot (\nabla \cdot \sigma_i).
\]
The claim is now a result of the skew-symmetry of $\sigma_i$, more precisely of 
\[
\nabla \partial_i u_{\mathrm{hom},1} \cdot (\nabla \cdot \sigma_i) = - \nabla \cdot (\sigma_i \nabla \partial_i u_{\mathrm{hom},1}).
\]
Testing \eqref{eq:twoscaleformula} with $z$ entails 
\begin{align*}
\int |\nabla z|_a^2 &\lesssim \int |g-g_1|_{a^{-1}}^2 + \int |\phi_i \nabla \partial_i u_{\mathrm{hom},1}|_a^2 + \int |\sigma_i \nabla \partial_i u_{\mathrm{hom},1}|_{a^{-1}}^2.
\end{align*}
For the first term on the right-hand side we aim to apply Poincar\'e's inequality noting that $g-g_1 \in W^{1,\frac{2q}{q-1}}(\mathbb R^d)$ is supported in $B_{R+1}$. Up to several constants depending in particular on $R$, we get 
\[
\int |g-g_1|_{a^{-1}}^2 \lesssim \Big( \int_{B_{R+1}} \big| a^{-1} \big|^q \Big)^\frac1q \Big( \int_{B_{R+1}} \big| \nabla g - \mathbb 1 \ast \nabla g \big|^\frac{2q}{q-1} \Big)^\frac{q-1}{q} \lesssim \Big( \int \big| \nabla g \big|^\frac{2q}{q-1} \Big)^\frac{q-1}{q}.
\]
% and that the average of $g-g_1$ over unit balls vanishes
For the second term, we note that $\nabla \partial_i u_{\mathrm{hom},1} = \frac{1}{|B|} \mathbb{1}_{B} \ast \nabla \partial_i u_\mathrm{hom}$ (with $B \coleq B_1(0)$) leading together with Jensen's inequality to 
\begin{align*} 
\int |\phi_i \nabla \partial_i u_{\mathrm{hom},1}|_a^2 \lesssim \int |a| |\phi_i|^2 \Big| \int_{B(x)} \nabla \partial_i u_{\mathrm{hom}} \Big|^2 \lesssim \int |a| |\phi|^2 \, \mathbb{1}_B \ast | \nabla^2 u_{\mathrm{hom}}|^2.
\end{align*}
Proceeding with H\"older's inequality, we deduce 
\begin{align*} 
\int |\phi_i \nabla \partial_i u_{\mathrm{hom},1}|_a^2 \lesssim \int \Big( \int_{B(x)} |a| |\phi|^2 \Big) | \nabla^2 u_{\mathrm{hom}}|^2 \leq \int \Big( \int_{B(x)} |a|^p \Big)^\frac1p \Big( \int_{B(x)} |\phi|^\frac{2p}{p-1} \Big)^\frac{p-1}{p} | \nabla^2 u_{\mathrm{hom}}|^2.
\end{align*}
Theorem \ref{theoremcorrectors} and $\fint_B \phi = 0$, therefore, result in 
\begin{align*} 
\int |\phi_i \nabla \partial_i u_{\mathrm{hom},1}|_a^2 \lesssim \int 4K r_e(x)^\frac{d}{p} \mathcal C(x)^2 \pi(|x|)^2 | \nabla^2 u_{\mathrm{hom}}|^2 \lesssim \mathcal C_{g,p}^2 \int \pi(|x|)^2 |\nabla g|^2, 
\end{align*}
where we introduced the random field 
\[
\mathcal C_{g,p}^2 \coleq \frac{\int \mathcal C(x)^2 r_e(x)^\frac{d}{p} \pi(|x|)^2 | \nabla^2 u_{\mathrm{hom}}|^2}{\int \pi(|x|)^2 |\nabla g|^2}.
\]
%\[
%\mathcal C_{\delta,g}(x) \coleq \frac{\Big( \int |\nabla g|^\frac{2q}{q-1} \Big)^\frac{q-1}{q} + \int \mathcal C(x)^2 \pi(|x|)^2 \big| \nabla \partial_i u_\mathrm{hom} \big|_a^2 + \int \mathcal C(x)^2 \pi(|x|)^2 \big| \nabla \partial_i u_\mathrm{hom} \big|_{a^{-1}}^2}{\Big( \int |\nabla g|^\frac{2q}{q-1} \Big)^\frac{q-1}{q} + \Big( \int \pi(|x|) |\nabla g|^\frac{2q}{q-1} \Big)^\frac{q-1}{q}},
%\]
Employing the stationarity of $\mathcal C(x)$ and $r_e(x)$ as well as a weighted $L^2$ estimate for $\nabla^2 u_\mathrm{hom}$ (recall that $\pi(\cdot)$ as defined in \eqref{eq:defpi} is a Muckenhoupt weight), we infer the bound  
\[
\big\langle \mathcal C_{g,p}^r \big\rangle \leq \bigg( \frac{\int \big\langle \mathcal C(x)^r r_e(x)^\frac{dr}{2p} \big\rangle^\frac2r \pi(|x|)^2 | \nabla^2 u_{\mathrm{hom}}|^2}{\int \pi(|x|)^2 |\nabla g|^2} \bigg)^\frac r2 \lesssim \Big\langle \mathcal C^r r_e^\frac{dr}{2p} \Big\rangle
\]
for any $r \geq 2$. Now let $r \geq \frac{2}{\varepsilon (1-\beta)} (1 + \frac{\alpha + 1}{\alpha} \frac{\varepsilon}{p})$ and observe that the previous estimate gives rise to 
\begin{align*}
	\Big\langle \mathcal C_{g,p}^{(1 + \frac{\alpha + 1}{\alpha} \frac{\varepsilon}{p})^{-1} \varepsilon (1-\beta) r} \Big\rangle &\lesssim \Big\langle \mathcal C^{(1 + \frac{\alpha + 1}{\alpha} \frac{\varepsilon}{p})^{-1} \varepsilon (1-\beta) r} r_e^{\frac{d}{2p}(1 + \frac{\alpha + 1}{\alpha} \frac{\varepsilon}{p})^{-1} \varepsilon (1-\beta) r} \Big\rangle \\ 
	&\leq \Big\langle \mathcal C^{\varepsilon (1-\beta) r} \Big\rangle^\frac{1}{1 + \frac{\alpha + 1}{\alpha} \frac{\varepsilon}{p}} \Big\langle r_e^{\frac{\alpha}{\alpha + 1} \frac{d}{2} (1-\beta) r} \Big\rangle^\frac{\frac{\alpha + 1}{\alpha} \frac{\varepsilon}{p}}{1 + \frac{\alpha + 1}{\alpha} \frac{\varepsilon}{p}} \lesssim (cr)^r
\end{align*}
using H\"older's inequality, the stochastic integrability of $\mathcal C$ and $r_e$ from Theorem \ref{theoremcorrectors} and Lemma \ref{lemma:re_moments}, respectively, and Lemma \ref{lemma:expmom_polymom}. This ensures the announced stretched exponential moment bounds for $\mathcal C_{g,p}$. The same arguments also show that 
\begin{align*} 
\int |\sigma_i \nabla \partial_i u_{\mathrm{hom},1}|_{a^{-1}}^2 \lesssim \mathcal C_{g,q}^2 \int \pi(|x|)^2 |\nabla g|^2 
\end{align*}
together with the random field 
\[
\mathcal C_{g,q}^2 \coleq \frac{\int \mathcal C(x)^2 r_e(x)^\frac{d}{q} \pi(|x|)^2 | \nabla^2 u_{\mathrm{hom}}|^2}{\int \pi(|x|)^2 |\nabla g|^2}
\]
allowing for the same stochastic integrability as $\mathcal C_{g,p}$ up to replacing $p$ by $q$. 
\end{proof}

\appendix
\section{Some auxiliary tools} 

\begin{lemma}[see e.g.\ {\cite[Lemma 6]{GNO20v3}}]
	\label{lemma:expmom_polymom}
	The following statements on a nonnegative random variable $F$ are equivalent.
	\begin{enumerate}
		\item There exists a constant $C \geq 1$ such that 
		\[
		\Big\langle \exp\Big( \frac1C F \Big) \Big\rangle < 2. 
		\]
		\item There exists some $p_0 \in \mathbb N$ and a constant $C \geq 1$ such that 
		\[
		\big\langle F^p \big\rangle^\frac1p \leq Cp
		\]
		for all $p \in \mathbb N$, $p \geq p_0$.
	\end{enumerate}
\end{lemma}

\begin{notation}
\label{defmatrixnorms}
For a matrix $M \in \mathbb R^{d \times d}$, we write 
\[
|M|_{F1} \coleq \sum_{i=1}^d \sum_{j=1}^d |M_{ij}|, \qquad |M|_{F\infty} \coleq \sup_{i \leq d} \sup_{j \leq d} |M_{ij}|.
\]
\end{notation}

\begin{lemma}
\label{lemmamatrixduality}
Let $M : \mathbb R^d \rightarrow \mathbb R^{d \times d}$ and suppose that 
\[
\sum_D \Big( \int_D M : a^{\frac12} N a^{\frac12} \Big)^2 \leq c^2 \sup_D \| |N|_{F\infty} \|_{L^\infty(D)}^2
\]
for some $c>0$ and all bounded $N : \mathbb R^d \rightarrow \mathbb R^{d \times d}$. Then, we have 
\[
\sum_D \big\| \big|a^{\frac 12} M a^{\frac 12}\big|_{F1} \big\|_{L^1(D)}^2 \leq c^2.
\]
\end{lemma}
\begin{proof}
We first observe that 
\[
\sum_D \| |M|_{F1} \|_{L^1(D)}^2 = \sup_{\genfrac{}{}{0pt}{1}{N \in \mathbb R^{d \times d}}{N \text{bounded}}} \frac{\sum_D \big( \int_D M : N \big)^2}{\sup_D \| |N|_{F\infty} \|_{L^\infty(D)}^2}
\]
for all matrix-valued functions $M, N \in \mathbb R^{d \times d}$ provided that $N_{ij}(\cdot)$ is bounded for all $i$ and $j$. This can be verified by elementary arguments from linear algebra. The previous identity, in particular, implies that the bound $\sum_D \| |M|_{F1} \|_{L^1(D)}^2 \leq c^2$ holds, if 
\[
\sum_D \Big( \int_D M : N \Big)^2 \leq c^2 \sup_D \| |N|_{F\infty} \|_{L^\infty(D)}^2
\]
for some $c>0$ and all bounded $N : \mathbb R^d \rightarrow \mathbb R^{d \times d}$. Replacing $M$ by $a^{\frac 12} M a^{\frac 12}$ and observing that $a^{\frac 12} M a^{\frac 12} : N = M : a^{\frac 12} N a^{\frac 12}$ allows to conclude. 
\end{proof}

\section*{Statements and Declarations}
The authors were partially supported by the German Science Foundation DFG in context of the first author's Emmy Noether
Junior Research Group BE 5922/1-1. The authors have no relevant financial or non-financial interests to declare that are relevant to the content of this article. Data sharing is not applicable to this article as no datasets were generated or analyzed during the current study. The authors thank Adolfo Arroyo-Rabasa for his contributions to the project in its very first stage.


\begin{bibdiv}
	\begin{biblist}
		\bib{ACS21}{article}{
			AUTHOR = {Andres, Sebastian}, 
			AUTHOR = {Chiarini, Alberto}, 
			AUTHOR = {Slowik, Martin},
			TITLE = {Quenched local limit theorem for random walks among
				time-dependent ergodic degenerate weights},
			JOURNAL = {Probab. Theory Related Fields},
			FJOURNAL = {Probability Theory and Related Fields},
			VOLUME = {179},
			YEAR = {2021},
			NUMBER = {3-4},
			PAGES = {1145--1181},
			ISSN = {0178-8051},
			MRCLASS = {60K37 (60F17 82C41 82C43)},
			MRNUMBER = {4242632},
			%DOI = {10.1007/s00440-021-01028-6},
			%URL = {https://doi.org/10.1007/s00440-021-01028-6},
		}
		
		\bib{ADS15}{article}{
			AUTHOR = {Andres, Sebastian}, 
			AUTHOR = {Deuschel, Jean-Dominique}, 
			AUTHOR = {Slowik, Martin},
			TITLE = {Invariance principle for the random conductance model in a
				degenerate ergodic environment},
			JOURNAL = {Ann. Probab.},
			FJOURNAL = {The Annals of Probability},
			VOLUME = {43},
			YEAR = {2015},
			NUMBER = {4},
			PAGES = {1866--1891},
			ISSN = {0091-1798},
			MRCLASS = {60K37 (60F17 82C41)},
			MRNUMBER = {3353817},
			MRREVIEWER = {Oriane Blondel},
			%DOI = {10.1214/14-AOP921},
			%URL = {https://doi.org/10.1214/14-AOP921},
		}
		
		\bib{AD18}{article}{
			AUTHOR = {Armstrong, Scott}, 
			AUTHOR = {Dario, Paul},
			TITLE = {Elliptic regularity and quantitative homogenization on
				percolation clusters},
			JOURNAL = {Comm. Pure Appl. Math.},
			FJOURNAL = {Communications on Pure and Applied Mathematics},
			VOLUME = {71},
			YEAR = {2018},
			NUMBER = {9},
			PAGES = {1717--1849},
			ISSN = {0010-3640},
			MRCLASS = {35B27 (35J25 60K35 82B43)},
			MRNUMBER = {3847767},
			MRREVIEWER = {Denis I. Borisov},
			%DOI = {10.1002/cpa.21726},
			%URL = {https://doi.org/10.1002/cpa.21726},
		}
		
		\bib{AS16}{article}{
			AUTHOR = {Armstrong, Scott N.}, 
			AUTHOR = {Smart, Charles K.},
			TITLE = {Quantitative stochastic homogenization of convex integral
				functionals},
			JOURNAL = {Ann. Sci. \'{E}c. Norm. Sup\'{e}r. (4)},
			FJOURNAL = {Annales Scientifiques de l'\'{E}cole Normale Sup\'{e}rieure. Quatri\`eme
				S\'{e}rie},
			VOLUME = {49},
			YEAR = {2016},
			NUMBER = {2},
			PAGES = {423--481},
			ISSN = {0012-9593},
			MRCLASS = {49J45 (49K45 60H15 60H30)},
			MRNUMBER = {3481355},
			MRREVIEWER = {Ada Bottaro Aruffo},
			%DOI = {10.24033/asens.2287},
			%URL = {https://doi.org/10.24033/asens.2287},
		}
	
		\bib{AKM19}{book}{
			AUTHOR = {Armstrong, Scott},
			AUTHOR = {Kuusi, Tuomo},
			AUTHOR = {Mourrat, Jean-Christophe},
			TITLE = {Quantitative stochastic homogenization and large-scale regularity},
			SERIES = {Grundlehren der mathematischen Wissenschaften [Fundamental Principles of Mathematical Sciences]},
			VOLUME = {352},
			PUBLISHER = {Springer, Cham},
			YEAR = {2019},
			PAGES = {xxxviii+518},
			ISBN = {978-3-030-15544-5; 978-3-030-15545-2; 978-3-030-15547-6},
			MRCLASS = {35-02 (35B27 60F17 60H15)},
			MRNUMBER = {3932093},
			%DOI = {10.1007/978-3-030-15545-2},
			%URL = {https://doi.org/10.1007/978-3-030-15545-2},
		}
		
		\bib{AL87}{article}{
			AUTHOR = {Avellaneda, Marco}, 
			AUTHOR = {Lin, Fang-Hua},
			TITLE = {Compactness methods in the theory of homogenization},
			JOURNAL = {Comm. Pure Appl. Math.},
			FJOURNAL = {Communications on Pure and Applied Mathematics},
			VOLUME = {40},
			YEAR = {1987},
			NUMBER = {6},
			PAGES = {803--847},
			ISSN = {0010-3640},
			MRCLASS = {35B99 (35J55 35R05)},
			MRNUMBER = {910954},
			MRREVIEWER = {Pierre-Louis Lions},
			%DOI = {10.1002/cpa.3160400607},
			%URL = {https://doi.org/10.1002/cpa.3160400607},
		}
		
		\bib{Bar04}{article}{
			AUTHOR = {Barlow, Martin T.},
			TITLE = {Random walks on supercritical percolation clusters},
			JOURNAL = {Ann. Probab.},
			FJOURNAL = {The Annals of Probability},
			VOLUME = {32},
			YEAR = {2004},
			NUMBER = {4},
			PAGES = {3024--3084},
			ISSN = {0091-1798},
			MRCLASS = {60K37 (60K35)},
			MRNUMBER = {2094438},
			MRREVIEWER = {Nikita Y. Ratanov},
			%DOI = {10.1214/009117904000000748},
			%URL = {https://doi.org/10.1214/009117904000000748},
		}
		
		\bib{BCF19}{article}{
			AUTHOR = {Bella, Peter}, 
			AUTHOR = {Chiarini, Alberto}, 
			AUTHOR = {Fehrman, Benjamin},
			TITLE = {A {L}iouville theorem for stationary and ergodic ensembles of
				parabolic systems},
			JOURNAL = {Probab. Theory Related Fields},
			FJOURNAL = {Probability Theory and Related Fields},
			VOLUME = {173},
			YEAR = {2019},
			NUMBER = {3-4},
			PAGES = {759--812},
			ISSN = {0178-8051},
			MRCLASS = {35B27 (35B53 35K45 60H25 60K37)},
			MRNUMBER = {3936146},
			MRREVIEWER = {Adrian Muntean},
			%DOI = {10.1007/s00440-018-0843-z},
			%URL = {https://doi.org/10.1007/s00440-018-0843-z},
		}
		
		\bib{BFO18}{article}{
			author={Bella, Peter},
			author={Fehrman, Benjamin},
			author={Otto, Felix},
			title={A {L}iouville theorem for elliptic systems with degenerate
				ergodic coefficients},
			date={2018},
			ISSN={1050-5164},
			journal={Ann. Appl. Probab.},
			volume={28},
			number={3},
			pages={1379\ndash 1422},
			url={https://doi.org/10.1214/17-AAP1332},
		}
		
		\bib{BO16}{article}{
			AUTHOR = {Bella, Peter}, 
			AUTHOR = {Otto, Felix},
			TITLE = {Corrector estimates for elliptic systems with random periodic
				coefficients},
			JOURNAL = {Multiscale Model. Simul.},
			FJOURNAL = {Multiscale Modeling \& Simulation. A SIAM Interdisciplinary
				Journal},
			VOLUME = {14},
			YEAR = {2016},
			NUMBER = {4},
			PAGES = {1434--1462},
			ISSN = {1540-3459},
			MRCLASS = {35B27 (35J47 35R60 60H15 74Q05)},
			MRNUMBER = {3573304},
			%MRREVIEWER = {Taras A. Mel\cprime nyk},
			%DOI = {10.1137/15M1037147},
			%URL = {https://doi.org/10.1137/15M1037147},
		}
		
		\bib{BS21}{article}{
			AUTHOR = {Bella, Peter}, 
			AUTHOR = {Sch\"{a}ffner, Mathias},
			TITLE = {Local boundedness and {H}arnack inequality for solutions of
				linear nonuniformly elliptic equations},
			JOURNAL = {Comm. Pure Appl. Math.},
			FJOURNAL = {Communications on Pure and Applied Mathematics},
			VOLUME = {74},
			YEAR = {2021},
			NUMBER = {3},
			PAGES = {453--477},
			ISSN = {0010-3640},
			MRCLASS = {35J30},
			MRNUMBER = {4201290},
			MRREVIEWER = {Yongqiang Fu},
			%DOI = {10.1002/cpa.21876},
			%URL = {https://doi.org/10.1002/cpa.21876},
		}
	
		\bib{BS22}{article}{
			AUTHOR = {Bella, Peter},
			AUTHOR = {Sch\"{a}ffner, Mathias},
			TITLE = {Non-uniformly parabolic equations and applications to the
				random conductance model},
			JOURNAL = {Probab. Theory Related Fields},
			FJOURNAL = {Probability Theory and Related Fields},
			VOLUME = {182},
			YEAR = {2022},
			NUMBER = {1-2},
			PAGES = {353--397},
			ISSN = {0178-8051},
			MRCLASS = {60K37 (35B65 35K65 60F17)},
			MRNUMBER = {4367950},
			%DOI = {10.1007/s00440-021-01081-1},
			%URL = {https://doi.org/10.1007/s00440-021-01081-1},
		}
	
		\bib{BDCKY15}{article}{
			AUTHOR = {Benjamini, Itai}, 
			AUTHOR = {Duminil-Copin, Hugo}, 
			AUTHOR = {Kozma, Gady}, 
			AUTHOR = {Yadin, Ariel},
			TITLE = {Disorder, entropy and harmonic functions},
			JOURNAL = {Ann. Probab.},
			FJOURNAL = {The Annals of Probability},
			VOLUME = {43},
			YEAR = {2015},
			NUMBER = {5},
			PAGES = {2332--2373},
			ISSN = {0091-1798},
			MRCLASS = {60K37 (31A05 37A35 60B15 60J10 82B43)},
			MRNUMBER = {3395463},
			MRREVIEWER = {Daniel Boivin},
			%DOI = {10.1214/14-AOP934},
			%URL = {https://doi.org/10.1214/14-AOP934},
		}
		
		\bib{CD16}{article}{
			AUTHOR = {Chiarini, Alberto}, 
			AUTHOR = {Deuschel, Jean-Dominique},
			TITLE = {Invariance principle for symmetric diffusions in a degenerate
				and unbounded stationary and ergodic random medium},
			JOURNAL = {Ann. Inst. Henri Poincar\'{e} Probab. Stat.},
			FJOURNAL = {Annales de l'Institut Henri Poincar\'{e} Probabilit\'{e}s et
				Statistiques},
			VOLUME = {52},
			YEAR = {2016},
			NUMBER = {4},
			PAGES = {1535--1563},
			ISSN = {0246-0203},
			MRCLASS = {60K37 (60F17)},
			MRNUMBER = {3573286},
			%MRREVIEWER = {Wolfgang L\"{o}hr},
			%DOI = {10.1214/15-AIHP688},
			%URL = {https://doi.org/10.1214/15-AIHP688},
		}
		
%		\bib{CG21+}{article}{
%			author={Clozeau, Nicolas},
%			author={Gloria, Antoine},
%			title={Quantitative nonlinear homogenization: control of oscillations},
%			date={2021},
%			note={arXiv:2104.04263v3},
%		}
	
		\bib{DS16}{article}{
			AUTHOR = {Deuschel, Jean-Dominique}, 
			AUTHOR = {Slowik, Martin},
			TITLE = {Invariance principle for the one-dimensional dynamic random
				conductance model under moment conditions},
			BOOKTITLE = {Stochastic analysis on large scale interacting systems},
			SERIES = {RIMS K\^{o}ky\^{u}roku Bessatsu, B59},
			PAGES = {69--84},
			PUBLISHER = {Res. Inst. Math. Sci. (RIMS), Kyoto},
			YEAR = {2016},
			MRCLASS = {60K37 (60F17 82C41)},
			MRNUMBER = {3675925},
		}
		
		\bib{DG20}{article}{
			author={Duerinckx, Mitia},
			author={Gloria, Antoine},
			title={Multiscale functional inequalities in probability: concentration
				properties},
			date={2020},
			journal={ALEA Lat. Am. J. Probab. Math. Stat.},
			volume={17},
			number={1},
			pages={133\ndash 157},
			url={https://doi.org/10.30757/alea.v17-06},
		}
		
		\bib{FO16}{article}{
			AUTHOR = {Fischer, Julian}, 
			AUTHOR = {Otto, Felix},
			TITLE = {A higher-order large-scale regularity theory for random
				elliptic operators},
			JOURNAL = {Comm. Partial Differential Equations},
			FJOURNAL = {Communications in Partial Differential Equations},
			VOLUME = {41},
			YEAR = {2016},
			NUMBER = {7},
			PAGES = {1108--1148},
			ISSN = {0360-5302},
			MRCLASS = {35J15 (35B27 35B65 35J47 35R60 60H25)},
			MRNUMBER = {3528529},
			MRREVIEWER = {Carmen Calvo-Jurado},
			%DOI = {10.1080/03605302.2016.1179318},
			%URL = {https://doi.org/10.1080/03605302.2016.1179318},
		}
	
		\bib{GM12}{book}{
			AUTHOR = {Giaquinta, Mariano},
			AUTHOR = {Martinazzi, Luca},
			TITLE = {An introduction to the regularity theory for elliptic systems, harmonic maps and minimal graphs},
			SERIES = {Appunti. Scuola Normale Superiore di Pisa (Nuova Serie) [Lecture Notes. Scuola Normale Superiore di Pisa (New Series)]},
			VOLUME = {11},
			EDITION = {Second},
			PUBLISHER = {Edizioni della Normale, Pisa},
			YEAR = {2012},
			PAGES = {xiv+366},
			ISBN = {978-88-7642-442-7; 978-88-7642-443-4},
			MRCLASS = {35-02 (35B65 35J20 35J60 58E20)},
			MRNUMBER = {3099262},
			%DOI = {10.1007/978-88-7642-443-4},
			%URL = {https://doi.org/10.1007/978-88-7642-443-4},
		}
		
		\bib{GT01}{book}{
			AUTHOR = {Gilbarg, David},
			AUTHOR = {Trudinger, Neil S.},
			TITLE = {Elliptic partial differential equations of second order},
			SERIES = {Classics in Mathematics},
			NOTE = {Reprint of the 1998 edition},
			PUBLISHER = {Springer-Verlag, Berlin},
			YEAR = {2001},
			PAGES = {xiv+517},
			ISBN = {3-540-41160-7},
			MRCLASS = {35-02 (35Jxx)},
			MRNUMBER = {1814364},
		}
		
		\bib{GNO20v2}{misc}{
			author={Gloria, Antoine},
			author={Neukamm, Stefan},
			author={Otto, Felix},
			title={A regularity theory for random elliptic operators},
			date={2014},
			note={arXiv:1409.2678v2},
		}
		
		\bib{GNO20v3}{misc}{
			author={Gloria, Antoine},
			author={Neukamm, Stefan},
			author={Otto, Felix},
			title={A regularity theory for random elliptic operators and
				homogenization},
			date={2015},
			note={arXiv:1409.2678v3},
		}
		
		\bib{GNO20}{article}{
			author={Gloria, Antoine},
			author={Neukamm, Stefan},
			author={Otto, Felix},
			title={A regularity theory for random elliptic operators},
			date={2020},
			ISSN={1424-9286},
			journal={Milan J. Math.},
			volume={88},
			number={1},
			pages={99\ndash 170},
			url={https://doi.org/10.1007/s00032-020-00309-4},
		}
		
		\bib{GNO21}{article}{
			AUTHOR = {Gloria, Antoine},
			AUTHOR = {Neukamm, Stefan},
			AUTHOR = {Otto, Felix},
			TITLE = {Quantitative estimates in stochastic homogenization for correlated coefficient fields},
			JOURNAL = {Anal. PDE},
			FJOURNAL = {Analysis \& PDE},
			VOLUME = {14},
			YEAR = {2021},
			NUMBER = {8},
			PAGES = {2497--2537},
			ISSN = {2157-5045},
			MRCLASS = {35J15 (35J47 60H25 74Q05)},
			MRNUMBER = {4377865},
			%DOI = {10.2140/apde.2021.14.2497},
			%URL = {https://doi.org/10.2140/apde.2021.14.2497},
		}
		
%		\bib{Iwa98}{article}{
%			author={Iwaniec, Tadeusz},
%			title={{The Gehring Lemma}},
%			pages={181\ndash 204},
%			book={
%				title = {{Quasiconformal Mappings and Analysis}},
%				subtitle = {{A Collection of Papers Honoring F.W. Gehring}},
%				editor = {Duren, Peter},
%				editor = {Heinonen, Juha},
%				editor = {Osgood, Brad},
%				editor = {Palka, Bruce},
%				publisher = {Springer},
%				address = {New York, NY},
%				date = {1998},
%				doi = {10.1007/978-1-4612-0605-7},
%				url = {https://doi.org/10.1007/978-1-4612-0605-7},
%			},
%		}

		\bib{Kre85}{book}{
			AUTHOR = {Krengel, Ulrich},
			TITLE = {Ergodic theorems},
			SERIES = {De Gruyter Studies in Mathematics},
			VOLUME = {6},
			NOTE = {With a supplement by Antoine Brunel},
			PUBLISHER = {Walter de Gruyter \& Co., Berlin},
			YEAR = {1985},
			PAGES = {viii+357},
			ISBN = {3-11-008478-3},
			MRCLASS = {28-02 (28Dxx 47A35)},
			MRNUMBER = {797411},
			MRREVIEWER = {E. Flytzanis},
			%DOI = {10.1515/9783110844641},
			%URL = {https://doi.org/10.1515/9783110844641},
		}

		\bib{MO15}{article}{
			AUTHOR = {Marahrens, Daniel}, 
			AUTHOR = {Otto, Felix},
			TITLE = {Annealed estimates on the {G}reen function},
			JOURNAL = {Probab. Theory Related Fields},
			FJOURNAL = {Probability Theory and Related Fields},
			VOLUME = {163},
			YEAR = {2015},
			NUMBER = {3-4},
			PAGES = {527--573},
			ISSN = {0178-8051},
			MRCLASS = {35B27 (35J08 35J25 39A70 60H25)},
			MRNUMBER = {3418749},
			%DOI = {10.1007/s00440-014-0598-0},
			%URL = {https://doi.org/10.1007/s00440-014-0598-0},
		}
		
		\bib{Sap17}{article}{
			AUTHOR = {Sapozhnikov, Artem},
			TITLE = {Random walks on infinite percolation clusters in models with
				long-range correlations},
			JOURNAL = {Ann. Probab.},
			FJOURNAL = {The Annals of Probability},
			VOLUME = {45},
			YEAR = {2017},
			NUMBER = {3},
			PAGES = {1842--1898},
			ISSN = {0091-1798},
			MRCLASS = {60K37 (58J35 60G15 60K35)},
			MRNUMBER = {3650417},
			MRREVIEWER = {Ji\v{r}\'{\i} \v{C}ern\'{y}},
			%DOI = {10.1214/16-AOP1103},
			%URL = {https://doi.org/10.1214/16-AOP1103},
		}
		
		\bib{Sim97}{article}{
			AUTHOR = {Simon, Leon},
			TITLE = {Schauder estimates by scaling},
			JOURNAL = {Calc. Var. Partial Differential Equations},
			FJOURNAL = {Calculus of Variations and Partial Differential Equations},
			VOLUME = {5},
			YEAR = {1997},
			NUMBER = {5},
			PAGES = {391--407},
			ISSN = {0944-2669},
			MRCLASS = {35H05 (35B45 35E20 35S15)},
			MRNUMBER = {1459795},
			MRREVIEWER = {Qing Yi Chen},
			%DOI = {10.1007/s005260050072},
			%URL = {https://doi.org/10.1007/s005260050072},
		}
		
		\bib{Tru71}{article}{
			AUTHOR = {Trudinger, Neil S.},
			TITLE = {On the regularity of generalized solutions of linear,
				non-uniformly elliptic equations},
			JOURNAL = {Arch. Rational Mech. Anal.},
			FJOURNAL = {Archive for Rational Mechanics and Analysis},
			VOLUME = {42},
			YEAR = {1971},
			PAGES = {50--62},
			ISSN = {0003-9527},
			MRCLASS = {35D10},
			MRNUMBER = {344656},
			MRREVIEWER = {V.-V. Olariu},
			%DOI = {10.1007/BF00282317},
			%URL = {https://doi.org/10.1007/BF00282317},
		}
	\end{biblist}
\end{bibdiv}
\end{document}